\documentclass[11pt]{amsart}

\usepackage[marginparwidth=75pt]{geometry}

\usepackage{amsmath, amsthm, amssymb}

\usepackage{palatino}         
\usepackage{enumitem}

\usepackage[abs]{overpic}		  

\usepackage{xcolor}
\definecolor{indigo}{rgb}{0.29, 0.0, 0.51}  
\definecolor{dred}{RGB}{237, 28, 36}
\usepackage[colorlinks, urlcolor=indigo, linkcolor=indigo, citecolor=indigo]{hyperref}

\theoremstyle{plain}
\newtheorem{theorem}{Theorem}
\newtheorem{corollary}[theorem]{Corollary}
\newtheorem{proposition}[theorem]{Proposition}
\newtheorem{lemma}[theorem]{Lemma}
\newtheorem{question}[theorem]{Question}
\newtheorem{conjecture}[theorem]{Conjecture}

\theoremstyle{definition}
\newtheorem{definition}[theorem]{Definition}

\theoremstyle{remark}
\newtheorem{remark}[theorem]{Remark}

\numberwithin{theorem}{section}

\newcommand{\dfn}[1]{{\em #1}}        
\newcommand{\Q}{\mathbb{Q}}           
\newcommand{\Z}{\mathbb{Z}}           

\makeatletter
\newcommand*\bigcdot{\mathpalette\bigcdot@{0.6}}
\newcommand*\bigcdot@[2]{\mathbin{\vcenter{\hbox{\scalebox{#2}{$\m@th#1\bullet$}}}}}
\makeatother

\DeclareMathOperator\tb{tb}                   
\DeclareMathOperator\rot{rot}                 

\begin{document}

\title {Complementary legs  and symplectic rational balls}

\author{John B. Etnyre}

\author{Burak Ozbagci}

\author{B\"{u}lent Tosun}

\address{Department of Mathematics \\ Georgia Institute of Technology \\ Atlanta \\ Georgia}

\email{etnyre@math.gatech.edu}

\address{Department of Mathematics \\ Ko\c{c} University \\ Istanbul \\ Turkey}

\email{bozbagci@ku.edu.tr}

\address{University of Alabama\\Tuscaloosa\\Alabama}

\email{btosun@ua.edu}

\subjclass[2020]{57R17, 57K33, 57K43}

\begin{abstract} 
We show that a small Seifert fibered space with complementary legs does not symplectically bound a rational homology ball for at least one choice of orientation. In the case $e_0\leq -1$, we characterize when a small Seifert fibered space with uniquely complementary legs symplectically bounds a rational homology ball. In the case $e_0\geq 0$, we characterize when a small Seifert fibered space with complementary legs,  equipped with a balanced contact structure, symplectically bounds a rational homology ball. Our results highlight a sharp contrast with the smooth category, where many more such Seifert fibered spaces are known to bound smooth rational homology balls. As a consequence of the results above, we also complete the classification of contact structures on oriented spherical $3$-manifolds that admit symplectic rational homology ball fillings. In particular, we show that a closed, oriented $3$-manifold with finite fundamental group admits at most six contact structures, up to isotopy, which are symplectically fillable by rational homology balls.
\end{abstract}

\maketitle

\section{Introduction}

There has been substantial work devoted to determining which rational homology $3$-spheres bound rational homology $4$-balls. In our recent work \cite{EtnyreOzbagciTosun2025}, we studied this question in the symplectic category for small Seifert fibered spaces. In particular, we proved that none of the contact structures on a small Seifert fibered space with $e_0\leq -5$ admit a symplectic rational homology ball filling, even though many such spaces smoothly bound rational homology balls. Earlier, Bhupal and Stipsicz \cite{BhupalStipsicz11}  characterized which Milnor fillable contact structures on small Seifert fibered spaces with $e_0\leq -2$ symplectically bound rational homology balls. In \cite{EtnyreOzbagciTosun2025}, we further showed that these are the only contact structures on small Seifert fibered spaces bounding symplectic rational homology balls when $e_0=-4,-3$ and also when $e_0=-2$ if the space was in the Bhupal-Stipsicz list. 

\begin{figure}[htb]{
\begin{overpic}
{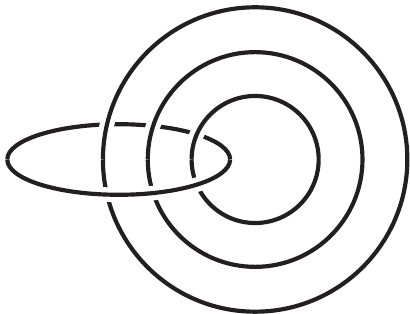}
\put(10, 55){$e_0$}
\put(140, 100){$-\frac 1{r_1}$}
\put(159, 113){$-\frac 1{r_2}$}
\put(179, 125){$-\frac 1{r_3}$}
\end{overpic}}
 \caption{A surgery diagram for the small Seifert fibered space $Y(e_{0};r_{1},r_{2},r_{3})$, with normalized Seifert invariants.}
  \label{fig:small}
\end{figure}
Here we continue our study by specifically focusing on small Seifert fibered spaces with complementary legs. Let $Y=Y(e_0; r_1, r_2, r_3)$ denote the small Seifert fibered space with normalized Seifert invariants, where $e_0 \in \Z$ and $r_i \in (0,1) \cap \mathbb{Q}$ for $i=1,2,3$, see Figure~\ref{fig:small}.
We say $Y$ has \dfn{complementary legs} if two of the $r_i$ add to $1$. Without loss of generality, we will assume, for the rest of the paper, that $r_1+r_3=1$. 

In \cite{EtnyreOzbagciTosun2025}, the authors were able to completely understand when a contact structure on a small Seifert fibered space with $e_0\leq -3$ bounded a symplectic rational homology ball, but for $e_0\geq -2$ only partial results were given. Specifically, constructions of and restrictions on such fillings were given. In this paper, using what is known about smooth rational homology ball fillings of Seifert fibered spaces with complementary legs, we can give a much more precise understanding of when contact structures on these manifolds admit symplectic rational homology ball fillings. In particular, we can verify some conjectures from \cite{EtnyreOzbagciTosun2025} in the case of Seifert fibered spaces with complementary legs. This improved understanding allows us to complete the study of contact structures on spherical $3$-manifolds that bound symplectic rational homology balls, which was begun in \cite{EtnyreOzbagciTosun2025}.

Lecuona \cite{Lecuona2019} has characterized precisely when a small Seifert fibered space with complementary legs smoothly bounds a rational homology ball. This will be the starting point for our analysis in the symplectic category. We will reprove her main result in a way that is more aligned with our symplectic geometric perspective in this paper. To this end, we first recall a result of Lisca. 

In \cite{Lisca2007}, Lisca showed that a lens space $L(p,q)$ bounds a rational homology ball if and only if $p/q\in \mathcal{R}$, where $\mathcal{R}$ is the collection of rational numbers $p/q>1$ with $gcd(p,q)=1$, $p=m^2$, and $q, p-q,$ or $q^*$ is of the form\begin{enumerate}
\item $mh\pm 1$ where $0<h<m$ with $gcd(h,m)=1$,
\item  $mh\pm 1$ where $0<h<m$ with $gcd(h,m)=2$,
\item $h(m\pm1)$ where $h>1$ divides $2m\mp 1$, or
\item $h(m\pm 1)$ where $h>1$ is odd and divides $m\pm 1$, 
\end{enumerate}
where $0<q^*<p$ is the inverse of $q$ mod $p$. We note that Item~(2) above did not appear in \cite{Lisca2007} though the proof there does produce it, see Remark~1.5 in \cite{BakerBuckLecuona16}. In Theorem~\ref{lecuonas} below, and throughout this paper, $[a_0, a_1, \ldots, a_n]$ denotes the (Hirzebruch-Jung) continued fraction expansion: 
\[
  a_0-\cfrac{1}{a_1-\cfrac{1}{a_2-\cfrac{1}{\cdots- \cfrac{1}{a_n}}}}
\]
We are now in a position to state a characterization of the small Seifert fibered spaces with complementary legs that bound rational homology balls.

\begin{theorem}[Lecuona 2019, \cite{Lecuona2019}]\label{lecuonas}
Let $Y=Y(e_0;r_1,r_2,r_3)$ be a small Seifert fibered space with complementary legs (i.e.,  $r_1+r_3=1$) whose surgery diagram is depicted in Figure~\ref{fig:small}.  Perform $(-e_0-1)$  Rolfsen twists on the $(-1/r_2)$-framed surgery curve to obtain a new surgery diagram of $Y$ such that the new framing on the horizontal curve is $-1$. The new framing on the $(-1/r_2)$-framed surgery curve  becomes $-1/r'_2$, where $r'_2= r_2+e_0+1$. Note that  $$-1/r'_2=-n + \frac{1}{[a^2_1, \ldots, a^2_{n_2}]}$$ for some uniquely determined $n \in \Z$, $n_2 \in \Z^{\geq 0}$ (here $n_2=0$ means that $-1/r'_2=-n \in \Z$), $a_1^2, \ldots, a_{n_2}^2 \in \Z$, with $a_i^2 \geq 2$ for $1 \leq i \leq n_2$ when $n_2 >0$. Then $Y$ smoothly bounds a rational homology ball if and only if either $n_2=0$ or $[a_1^2,\ldots, a^2_{n_2}]\in \mathcal{R}$. 
\end{theorem}

 Note that the Rolfsen-twist bookkeeping in Theorem~\ref{lecuonas} can be eliminated by setting
\[
  x \;=\; -\frac{1}{\,r_2+e_0+1\,}\,.
\]
Since $x=-1/r_2'$ is exactly the new framing on the $(-1/r_2)$-framed surgery curve, and $-n=\lfloor x\rfloor$, it follows that  $Y$ smoothly bounds a rational homology ball if and only if
\[
  x-\lfloor x\rfloor=0 \qquad\text{or}\qquad \bigl(x-\lfloor x\rfloor\bigr)^{-1}\in\mathcal R .
\]

\subsection{Symplectic rational homology balls}
We now investigate the existence of symplectic rational homology ball fillings of small Seifert fibered spaces with complementary legs.
\begin{theorem}\label{thm: smallSFSnotfillable} A small Seifert fibered space $Y(e_0; r_1, r_2, r_3)$ with complementary legs and $e_0\leq -2$ does not admit a rational homology ball symplectic filling. 
 \end{theorem} 
Since any small Seifert fibered space $Y(e_0; r_1, r_2, r_3)$, using the fact that $e_0(Y)+e_0(-Y)=-3$, satisfies $e_0 \leq -2$, up to orientation, we obtain the following immediate consequence.
\begin{corollary} A small Seifert fibered space with complementary legs does not symplectically bound a rational homology ball, with at least one orientation.
\end{corollary}
\begin{remark} We would like to emphasize that Theorem~\ref{thm: smallSFSnotfillable} confirms our Conjecture 1.5 in \cite{EtnyreOzbagciTosun2025}, for any small Seifert fibered space $Y(-2; r_1, r_2, r_3)$  with complementary legs. 
\end{remark}
\begin{remark} We would like to point out that when $e_0 \leq -3$, Theorem~\ref{thm: smallSFSnotfillable} also follows directly from  \cite[Theorem 1.1]{EtnyreOzbagciTosun2025} since one can easily verify that the small Seifert fibered spaces described by the plumbing diagrams in $\mathcal{QHB}$ (depicted in \cite[Figure 3]{EtnyreOzbagciTosun2025}) do not have complementary legs.   Any small Seifert fibered space $Y(-2; r_1, r_2, r_3)$ with complementary legs is an $L$-space and therefore,  \cite[Corollary 1.6]{EtnyreOzbagciTosun2025} implies that $Y(-2; r_1, r_2, r_3)$ with a pair of complementary legs admits {\em at most four} contact structures that might admit rational homology ball symplectic fillings. Hence, the results in our earlier work \cite{EtnyreOzbagciTosun2025} almost prove Theorem~\ref{thm: smallSFSnotfillable}  here,  but fall short when $e_0=-2$. 
\end{remark} 
\begin{remark}
We note that from Theorem~\ref{lecuonas} there are many small Seifert fibered spaces with complementary legs and any $e_0$ that bound smooth rational homology balls. So we see a stark contrast between the smooth and symplectic categories. 
\end{remark}

We now turn to the case $e_0 \geq -1$. In Lemma~4.5 of \cite{EtnyreOzbagciTosun2025}, several symplectic rational homology ball fillings of small Seifert fibered spaces $Y(-1;r_1,r_2,r_3)$ were constructed.\footnote{The Seifert invariants in \cite{EtnyreOzbagciTosun2025} appear different than the ones given in our theorem here, but substituting $h$ for $m-h$ and $n$ for $n+1$ equates the results there and in this paper. The discrepancy comes from the different surgery descriptions utilized in the two papers.} Here, we prove that these are the only such fillings of contact structures on  $Y(-1; r_1, r_2, r_3)$ with uniquely complementary legs. To define \dfn{uniquely complementary}, we first define a subleg. Using the slam-dunk move on surgery diagrams of $3$-manifolds, see \cite{EtnyreOzbagciTosun2025}, we can expand each of the circles with framing $1/r_i$ in Figure~\ref{fig:small} as a chain of unknots with integer surgery coefficients coming from the continued fraction expansion of $-1/r_i$.  We call this chain a \dfn{leg} of the Seifert fibered space. An \dfn{initial subleg} is a connected subset of this chain that is linked to the unknot with framing $e_0$. Now, a small Seifert fibered space is \dfn{uniquely complementary} if two of its legs are complementary and the third leg does not have an initial subleg that is complementary with an initial subleg of either of the other two legs. 

\begin{theorem} \label{thm: minusone} 
A small Seifert fibered space with uniquely complementary legs and $e_0=-1$,  carries a contact structure that is symplectically fillable by a rational homology ball if and only if it is of the form $$Y(-1; \dfrac{b}{a}, \dfrac{m^2}{nm^2-mh+1}, 1-\dfrac{b}{a})$$for some $a/b>1$ in reduced terms, some integer $n > 1$, and some relatively prime integers $0 < h < m$, or $n\geq 1$, $m=1$ and $h=0$.

Moreover, there are at least $\min\{\lfloor a/2\rfloor,n\}$, respectively $\min\{\lfloor a/2\rfloor,2n\}$, pairwise nonisotopic tight contact structures with this property on such a small Seifert fibered space if $m=1$, respectively $m>1$. If $2n$, respectively $n$, is less than or equal to $a/2$, then there are exactly $2n$, respectively $n$,  such contact structures when $m>1$, respectively $m=1$. Furthermore, if $n>2$, then for $2n-2$, respectively $n-2$, of these contact structures, there is a unique symplectic rational homology ball filling of the contact manifold, up to blowups. 
\end{theorem}

\begin{remark}
We expect the first paragraph in Theorem~\ref{thm: minusone}, and an analog of the second paragraph, to hold for any small Seifert fibered space with complementary legs (and not just uniquely complementary legs), but this will require further investigation. We note in many examples that this is indeed true. See, for example, Theorem~\ref{classifyfillings} and its proof. 
\end{remark}

\begin{remark}
The lower bound at the beginning of the second paragraph of Theorem~\ref{thm: minusone} can definitely be improved in some cases. For example, the proof of the theorem will show that if $a$ is odd, $a$ and $n$ are relatively prime, $m=1$, and $n>a^2$, then there are at least $a^2$ contact structures that bound symplectic rational homology balls. This follows because the first homology of the $3$-manifold is $\Z/a^2\Z$ and $a^2$ of the contact structures constructed in the proof will have distinct $\Gamma$-invariants as defined in \cite{Gompf98}. 
\end{remark}

\begin{remark}\label{extra}
We now consider some explicit examples where the bound on the number of contact structures bounding symplectic rational homology balls in Theorem~\ref{thm: minusone} is not sharp.
\begin{enumerate} 
\item  Consider $M_p=Y(-1; \frac{1}{2}, \frac{1}{p}, \frac{1}{2})$ for $p\geq 2$. With the notation of Theorem~\ref{thm: minusone},  this corresponds to $a/b=2, m=1, h=0$, and $n=p-1$.  According to Theorem~\ref{thm: minusone}, for any $p \geq 2$, there is at least one contact structure on $M_p$ that admits a symplectic rational homology ball filling. Note that by \cite[Remark~3.12]{GhigginiLiscaStipsicz07}, $M_p$ carries exactly three pairwise nonisotopic tight contact structures, for any $p\geq 2$. When $p>2$, one of those cannot be filled by a symplectic rational homology ball since it has $\theta\neq -2$. One may see this by combining Proposition~3.2 and Lemma~3.5 of \cite{GhigginiLiscaStipsicz07} and recalling the relation between the $d_3$ and $\theta$ invariants (that is, $4d_3-2=\theta$). The other two contact structures are contactomorphic and fillable by symplectic rational homology balls (as can be seen in our proof of the above theorem). 
 When $p=2$, however,  all three contact structures are pairwise contactomorphic (\cite[Section~$3.3$]{GhigginiLiscaStipsicz07}) and symplectically fillable by rational homology balls.

\item As another example, consider $Y_3=Y(-1; \frac{2}{3},  \frac{1}{2}, \frac{1}{3})$. According to Theorem~\ref{thm: minusone}, there is at least one tight contact structure that admits a symplectic rational homology ball filling. However, \cite{GhigginiLiscaStipsicz07} shows that there are exactly five tight contact structures on $Y_3$, up to isotopy. It turns out that $\theta \neq -2$ for two of these and hence they are not symplectically fillable by rational homology balls. In Figure~\ref{fig:3QHB}, we depicted three Stein rational homology balls with boundary $Y_3$. To see that the contact structures induced on $Y_3$ by these three diagrams are pairwise nonisotopic, we first notice that if $X$ denotes the underlying smooth rational homology ball, then $H^2(X)=\Z/3\Z$ and it is generated by the Poincar\'e dual of the homology class of the cocore $[C]$ to the $2$-handle. By \cite[Theorem~4.12]{Gompf98} we see that in any one of the Stein diagrams, this class is given by $\rot(L)[C]$ where $L$ is the Legendrian knot along which the $2$-handle is attached. Thus, the three lower diagrams give the Chern classes Poincar\'e dual to $2[C], 0,$ and $-2[C]$. Since these classes are all distinct in $H^2(X)$ we see by \cite[Theorem~1.2]{LiscaMatic97} that the three contact structures on $Y_3$, induced by these Stein rational homology balls, are pairwise nonisotopic.  Therefore, up to isotopy,  there are exactly three tight contact structures on $Y_3$ that are symplectically fillable by rational homology balls.
\end{enumerate}
\end{remark}

\begin{figure}[htb]{
\begin{overpic}
{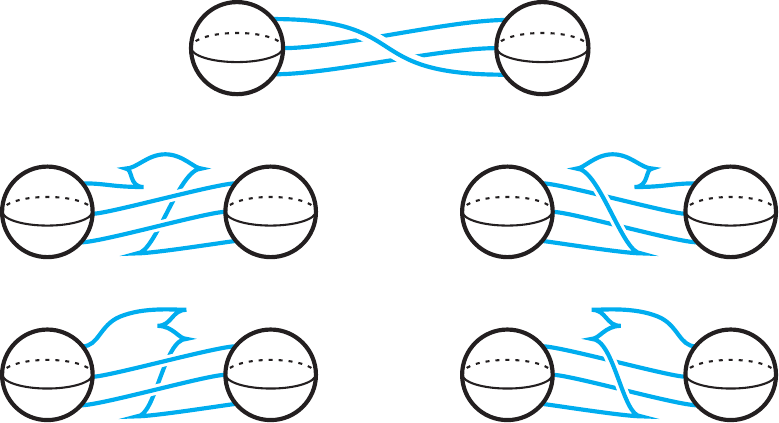}
\put(100, 137){\large\rotatebox{45}{$\cong$}}
\put(264, 142){\large\rotatebox{-45}{$\cong$}}
\end{overpic}}
 \caption{The Legendrian knots in the top three figures are Legendrian isotopic
(see \cite[Figure~20]{Gompf98}), and hence all three diagrams determine
the same Stein rational homology ball filling a fixed contact structure
on $Y_3$. The two Stein rational homology
balls depicted in the bottom row determine two more 
contact structures on $Y_3$, up to isotopy.}
  \label{fig:3QHB}
\end{figure}

According to \cite{GhigginiLiscaStipsicz06} we know that any tight contact structure on $Y=Y(e_0;r_1,r_2,r_3)$ with $e_0\geq 0$ is obtained from contact surgery on the link $L_1\cup L_2\cup L_3$ in $S^1\times S^2$ shown in Figure~\ref{fig:s1s2surgery}. We call each of the $L_i$ a leg of the Seifert fibration. Recall that the contact surgery on each of the $L_i$ is determined by a sequence of signs related to the continued fraction expansion of $-1/r_i$. We say that the contact surgery on leg $L_i$ is \dfn{consistent} if all the signs are the same. If $Y$ has complementary legs, then we say that the contact structure coming from contact surgery is \dfn{balanced} if each of the complementary legs is consistent, but they have opposite signs. 

In Lemma~4.5 of \cite{EtnyreOzbagciTosun2025}, several symplectic rational homology ball fillings of $Y$ with $e_0\geq 0$ were constructed. Our next result below shows that, for Seifert fibered spaces with complementary legs, these fillings account for {\em almost} all fillings that could exist. 

\begin{theorem}  \label{thm: nonneg}  
Let $Y$ be a small Seifert fibered space with complementary legs and $e_0\geq 0$, equipped with a balanced contact structure $\xi$. Then $(Y, \xi)$  bounds a symplectic rational homology ball if and only if $Y$ is of the form $Y(e_0;r, s, 1-r)$, where $s$ is given by 
\[
-\dfrac{1}{s}=\frac{-m^2+mh-1}{m^2+(e_0+1)(-m^2+mh-1)}<-1,
\]
for some $0<h<m$ relatively prime (note that once $e_0 \geq 0$ is fixed, only certain values  of $h$ and $m$ will satisfy the desired inequality). Moreover, on such a $Y$, there are four balanced contact structures (up to isotopy) that bound a symplectic rational homology ball. 

If $\xi$ is not a balanced contact structure, then $Y$ might bound a symplectic rational homology ball if it is given by the surgery diagram in Figure~\ref{fig:balanced} where $[a^2_1,\ldots, a^2_{n_2}]\in \mathcal{R}$. Given such a $Y$ there are at most two, respectively four, contact structures (up to isotopy) that can bound a symplectic rational homology ball if $e_0=0$, respectively $e_0>0$. 
\end{theorem}
\begin{remark}
We conjecture that none of the contact structures in the second paragraph of the theorem admit symplectic rational homology ball fillings. If so, then the theorems above completely characterize when a Seifert fibered space with complementary legs bounds a symplectic rational homology ball. The result below gives some evidence towards this conjecture, and also answers positively the prediction made in \cite[Remark~$1.21$]{EtnyreOzbagciTosun2025}. (See also Section~\ref{spherical}.)
\end{remark}

\begin{proposition}\label{dihedO}
A small  Seifert fibered space $Y=Y(e_0; \frac{1}{2}, s, \frac{1}{2})$ with $e_0\geq 0$ equipped with a contact structure $\xi$ admits a rational homology ball symplectic filling if and only if $Y$ is of the form as in the first part of Theorem~\ref{thm: nonneg} and $\xi$ is balanced.
\end{proposition}

\begin{remark}
We point out a subtle difference between characterizing smooth rational homology ball fillings versus symplectic rational homology ball fillings. It is a standard fact that the lens space $L(p,q)$ arises as the double cover of $S^3$ branched over the two-bridge link $K(p,q)$. In \cite[Corollary~$1.3$]{Lisca2007}, Lisca not only characterized which lens spaces bound smooth rational homology balls but also proved that $L(p,q)$, assuming $p$ is odd, bounds a rational homology ball exactly when $K(p,q)$ is a smoothly slice knot. In particular, a rational homology ball that $L(p,q)$ bounds can be constructed by taking the double branched cover of $B^4$ branched over the slice disk. This result was later extended by Lecuona \cite{Lecuona2019} to cover Seifert fibered spaces with complementary legs that arise as the double cover of $S^3$ branched along a Montesinos link that can be put in the form $K(p,q)\sqcup U$ where $U$ is the unknot. When determining symplectic rational homology ball fillings for lens spaces or Seifert fibered spaces with complementary legs, a natural guess would be that perhaps the two-bridge knots involved are symplectically or Lagrangian slice, which would require the knot to be at least quasi-positive.  It is well-known, by combining \cite{BoileauRudolph1995, KronheimerMrowka93, Rudolph93} that this is never the case. (See also \cite{Ozbagci-bridge}, for two new proofs of this fact). 
\end{remark}

The strategy of proof for Theorem~\ref{thm: smallSFSnotfillable} involves the $\theta$-invariant calculations. The $\theta$-invariant is an invariant of homotopy classes of plane fields, defined in \cite{Gompf98}. It is well-known, see for example \cite{EtnyreOzbagciTosun2025}, that if a contact $3$-manifold $(M,\xi)$ symplectically bounds a rational homology ball, then $\theta(\xi)=-2$. In Sections~\ref{m2case} and~\ref{subsec: theta} we show by explicit calculations that any contact structure on $Y(e_0;r_1,r_2,r_3)$ with complementary legs and $e_0\leq -2$ has $\theta$-invariant larger than $-2$. We prove Theorem~\ref{thm: minusone} (see Theorem~\ref{maine00}) and~\ref{thm: nonneg} (see Theorem~\ref{maine0+}) in Section~\ref{eogeqm1}, by constructing Stein cobordisms from lens spaces to certain small Seifert fibered spaces and using facts from \cite{EtnyreOzbagciTosun2025}. Section~\ref{sfswcl} discusses small Seifert fibered spaces with complementary legs and relates prior work in the area to conventions needed here. 

This paper can be thought of as a continuation of \cite{EtnyreOzbagciTosun2025}. As such, we refer to Section~2 of that paper for the background on standard facts from contact geometry, smooth and contact surgeries, and homotopy invariants of plane fields. 

\subsection{Spherical $3$-manifolds}\label{spherical}
We now turn to the classification of spherical $3$-manifolds {\em with either orientations}, equipped with some contact structures,  that admit rational homology ball symplectic fillings. We refer the reader to our earlier work \cite{EtnyreOzbagciTosun2025}, for basic definitions of spherical $3$-manifolds  that will be used below. 

It was shown by Choe and Park \cite{ChoePark2021}  that a spherical $3$-manifold $Y$ bounds a smooth rational homology ball if and only if $Y$ or $-Y$ is homeomorphic to one of the following manifolds: 
\begin{enumerate}
\item $L(p, q)$ such that $p/q \in\mathcal{R}$,
\item\label{2} $D(p, q)$ such that $q=p-1 >0$ or  $(p-q)/q' \in\mathcal{R}$, where $p$ and $q$ are relatively prime integers such that $0 < q < p$, and $0 < q' < p - q$ is
the reduction of $q$ modulo $p - q$,
\item\label{3} $T_3$, $T_{27}$ and $I_{49}$.
\end{enumerate}
See the beginning of this introductory section for the description of the set $\mathcal{R}$ mentioned in item (2). It follows that when searching for spherical $3$-manifolds bounding symplectic rational homology balls, we only need to consider the $3$-manifolds listed above.  

Suppose that  $\xi$ is a contact structure on the spherical $3$-manifold $Y$, oriented as the link of the corresponding quotient surface singularity. In \cite{EtnyreOzbagciTosun2025}, we showed that if $(Y, \xi)$ admits a symplectic rational homology ball filling,  then $Y$ is orientation-preserving diffeomorphic to a lens space $L(m^2,mh-1)$ for some coprime integers $0< h < m$, and $\xi$ is contactomorphic to the canonical contact structure $\xi_{can}$. We also observed, however, that this result does not {\em necessarily} hold true for a spherical $3$-manifold  equipped with the orientation opposite to the canonical one it carries when viewed as the singularity link. In particular, we showed that, among the spherical $3$-manifolds listed in item (3) above, $-T_3=Y(-1; 2/3, 1/2, 1/3)$ carries at least two nonisotopic tight contact structures, both of which admit symplectic rational homology ball fillings, but neither $-T_{27}=Y(3; 2/3, 1/2, 1/3)$ nor $-I_{49}=Y(0; 4/5, 1/2, 1/3)$ admits a symplectic rational homology ball filling. 
Here we observe that $-T_3$ actually carries exactly three pairwise nonisotopic contact structures which admit symplectic rational homology ball fillings. This is proven in Part~(2) of Remark~\ref{extra} and the Stein rational homology balls are depicted in Figure~\ref{fig:3QHB}, where $Y_3=-T_3$. 
  
 As a byproduct of our results on small Seifert fibered spaces,  we also complete the classification of all spherical $3$-manifolds with either orientations,  which admit rational homology ball symplectic fillings, by resolving  
the remaining case of ``oppositely oriented" dihedral-type spherical $3$-manifolds. 
\begin{theorem}\label{classifyfillings}
The only oriented spherical $3$-manifolds, equipped with some contact structures, which admit symplectic rational homology ball fillings are
\begin{enumerate}
\item $L(m^2, mh-1)$ for relatively prime $0<h<m$, with the two contact structures $\pm \xi_{can}$, or
\item $-D((n+1)m^2-mh+1, nm^2-mh+1)$ for relatively prime $0<h<m$ and $n \geq 1$, or $m=1$, $h=0$, and $n\geq 1$ with 
\begin{enumerate}
\item one of $6$ possible contact structures when $0<h<m$ and $n=2$,
\item one of $4$ possible contact structures when $0<h<m$ and $n=1$ or $n>2$,
\item one of $3$ possible contact structures when $m=1, h=0$ and $n=1$,
\item one of $2$ possible contact structures when $m=1, h=0$ and $n>1$, or
\end{enumerate}
\item $-T_3=Y(-1;2/3, 1/2, 1/3)$ with one of $3$ possible contact structures,
\end{enumerate} 
where the contact structures are counted up to isotopy.
\end{theorem}

Since the class of spherical $3$-manifolds coincides with the class of closed orientable $3$-manifolds with finite fundamental group, we immediately obtain the following corollary.

\begin{corollary} A closed, oriented  $3$-manifold with finite fundamental group
admits at most six contact structures, up to isotopy, which are
symplectically fillable by rational homology balls.
\end{corollary}

\subsection{Final remarks and questions} 
Our work in \cite{EtnyreOzbagciTosun2025} and in this paper strongly indicates a positive answer to the following conjecture.

\begin{conjecture} Let $Y=Y(e_0; r_1, r_2,  \ldots, r_n)$ be a Seifert fibered space that is also a rational homology sphere, whose 
associated star-shaped plumbing graph $\Gamma_Y$ is negative definite. Then $Y$
carries a contact structure $\xi$ such that $(Y, \xi)$ admits rational homology ball symplectic filling if and only if $n\leq 4$ and $Y$ is the link of a complex surface singularity, whose resolution graph belongs to a finite list of families provided by Bhupal and Stipsicz \cite[Figure 1 and Figure 2]{BhupalStipsicz11}, and $\pm \xi$ is the canonical contact structure.
\end{conjecture}

\begin{remark}
We note that our results of Theorem~\ref{thm: minusone} and Theorem~\ref{thm: nonneg} show that there are Seifert fibered spaces that have non-negative definite plumbing graphs but yet admit rational homology ball symplectic fillings. We also note that in some cases (see, for example, the second part of Theorem~\ref{thm: minusone}) this is achieved by arbitrarily many distinct tight contact structures. Although the examples and construction in Theorem~\ref{thm: minusone} and Theorem~\ref{thm: nonneg} produce small Seifert fibered spaces, in the upcoming paper we will extend our results to find Seifert fibered spaces with arbitrarily many singular fibers (but with controlled integer Euler number) that admit rational homology ball symplectic fillings. With that we will be able to provide a much general conjectural picture for when a Seifert fibered space $Y=Y(e_0; r_1, r_2,  \ldots, r_n)$ admits a symplectic rational homology ball filling.           
\end{remark}

We end with an interesting remark about orientations. As a consequence of our results in Theorem~\ref{thm: smallSFSnotfillable} and \ref{thm: minusone}, we obtain many examples of a small Seifert fibered space with no rational homology ball symplectic filling with either orientation, as well as many that have rational homology ball fillings with one orientation. We are not aware of a 3-manifold, other than $S^3$, that has a rational homology ball symplectic filling with both orientations. So we ask:

\begin{question} Suppose $Y$ is a closed, connected and oriented  3-manifold. Is it true that if $Y$ and $-Y$ both admit rational homology ball symplectic fillings, then $Y=S^3$?
\end{question}

\vspace{.2in}
\noindent
{\bf Acknowledgments:} We thank an anonymous referee for their careful reading and improvements to the paper. The first author was partially supported by National Science Foundation grant DMS-2203312 and the Georgia Institute of Technology's Elaine M. Hubbard Distinguished Faculty Award. The third author was supported in part by grants from the National Science Foundation (DMS-2105525 and CAREER DMS 2144363) and the Simons Foundation (636841, BT and 2023 Simons Fellowship). He also acknowledges the support by the Charles Simonyi Endowment at the Institute for Advanced Study.

\section{Small Seifert fibered spaces with complementary legs} \label{sfswcl}
Recall the Seifert fibered space $Y(e_0;r_1,r_2,r_3)$ has complementary legs if there are distinct $i$ and $j$ such that $r_i+r_j=1$. We will order the singular fibers such that $r_1+r_3=1$. In this section, we will consider the smooth topology and the contact topology of such Seifert fibered spaces.
\subsection{Smooth Seifert fibered spaces with complementary legs}
We begin by describing convenient surgery diagrams for small Seifert fibered spaces with complementary legs.
\begin{lemma}\label{gencomplementarylegs}
Any small Seifert fibered space with complementary legs is obtained from $S^1\times S^2$ by Dehn surgery on a regular fiber of some Seifert fibration of $S^1\times S^2$ and has a surgery diagram given in Figure~\ref{fig:balanced}. Here  $a/b=[a^1_1,\ldots, a^1_{n_1}] >1$, $p/q=[a^2_1,\ldots, a^2_{n_2}] >1$ (or  $n_2=0$ and there is no chain of unknots with surgery coefficients $-a^2_1, \ldots, -a^2_{n_2}$), and $a/(a-b)=[a^3_1,\ldots, a^3_{n_3}]$, all the $a_i^j\geq 2$, and $n\in \Z$. Moreover, any such Seifert fibered space is a rational homology sphere. 
\end{lemma}

\begin{figure}[htb]{
\begin{overpic}
{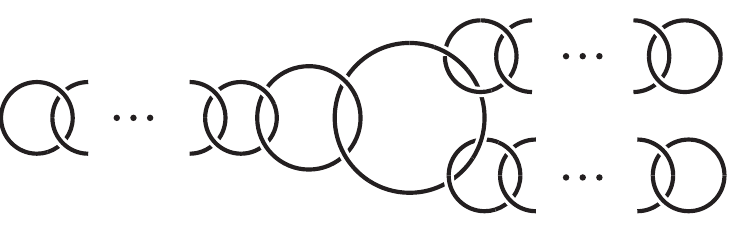}
\put(103, 77){$-a_1^2$}
\put(5, 77){$-a^2_{n_2}$}
\put(138, 82){$-n$}
\put(165, 10){$-1$}
\put(220, 105){$-a^1_1$}
\put(319, 105){$-a^1_{n_1}$}
\put(220, -5){$-a^3_1$}
\put(319, -5){$-a^3_{n_3}$}
\end{overpic}}
 \caption{Surgery diagram for a small Seifert fibered space with complementary legs. Here $a/b=[a^1_1,\ldots, a^1_{n_1}]>1$, $p/q=[a^2_1,\ldots, a^2_{n_2}] >1$ (or  $n_2=0$ and there is no chain of unknots with surgery coefficients  $-a^2_1, \ldots, -a^2_{n_2}$), and $a/(a-b)=[a^3_1,\ldots, a^3_{n_3}]$. }
  \label{fig:balanced}
\end{figure}
\begin{remark}\label{rem:nvalue}
We note that the Seifert fibered space in Figure~\ref{fig:balanced} is  in normalized form iff $n \geq 2$. Using Rolfsen twists, one may easily show 
\[
\begin{array}{l l}
e_0\geq 0 &\text{ if } n=1\\
e_0=-1 &\text{ if } n\geq 2\\
e_0=-2 &\text{ if } n\leq -1\\
e_0\leq -3 &\text{ if } n=0.
\end{array}
\]
In fact, if $-n+q/p$ is in the interval $(-1/{k}, -1/(k+1))$ for some integer $k \geq 1$ (which requires that $n=1$), then $e_0=k-1$ and if $-n+q/p$ is in $(1/(k+1), 1/k)$ for some integer $k \geq 1$ (which requires that $n=0$), then $e_0=-k-2$.
\end{remark}

\begin{proof}
Given a small Seifert fibered space $Y=Y(e_0; r_1, r_2, r_3)$ in normalized form, we see it has complementary legs if $-1/r_1=-a/b$ and $-1/r_3= -a/(a-b)$ and both quantities will be less than $-1$. We can perform a Rolfsen twist on the singular fiber $-a/(a-b)$ so that the singular fiber is now a singular fiber labeled $a/b$, and the $e_0$ term has been increased by $1$. 
\begin{figure}[htb]{
\begin{overpic}
{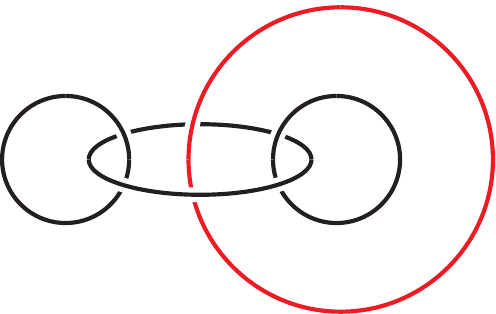}
\put(24, 111){$a/b$}
\put(148, 111){$-a/b$}
\put(210, 140){\color{dred}$F$}
\put(35, 67){$0$}
\end{overpic}}
 \caption{A surgery diagram for $S^1\times S^2$ with a Seifert fibered structure where $F$ is a regular fiber.}
  \label{fig:s1s2}
\end{figure}
We may now perform Rolfsen twists on the second singular fiber so that the central curve has surgery coefficient $0$. If we remove the second singular fiber from the surgery diagram for $Y$, then the resulting manifold is $S^1\times S^2$, and the curve that one would surger to obtain $Y$ is a regular fiber in a Seifert fibered structure on $S^1\times S^2$. If this is not clear, see \cite[Lemma~4.1]{EtnyreOzbagciTosun2025}. So $Y$ is obtained from Figure~\ref{fig:s1s2} by Dehn surgery on $F$. This establishes the first claim of the lemma. We note for future reference that the framing of $F$ given by the Heegaard torus on which $F$ sits agrees with the $0$-framing in the figure. 

For the second claim, we note that we can undo the Rolfsen twist on the $a/b$ fiber to get back to a $-a/(a-b)$ fiber and now the central curve has coefficient $-1$. Moreover, a continued fraction expansion of the surgery coefficient will be of the form 
\[
-n-\frac{1}{[-a_1^2,\ldots, -a_{n_2}^2]} = -n+\frac{1}{[a_1^2,\ldots, a_{n_2}^2]} ,
\] 
for some integer $n$ and integers $a^2_i \ge 2$ when $n_2 > 0$ and otherwise, when $n_2 = 0$, there is no such chain and the coefficient is
just the integer $-n$. This establishes the second claim of the
lemma. 

Finally, recall that a small Seifert fibered space $Y(e_0; r_1, r_2, r_3)$ in normal form is a rational homology sphere if and only if $e_0+r_1+r_2+r_3\not=0$. Note that 
$e_0 \in \Z$, and  $r_i\in(0,1)$, by definition. The last claim of the lemma follows since we have $r_1+r_3=1$ in our case. 
\end{proof}

We characterize when a small Seifert fibered space with complementary legs bounds a rational homology ball, but first recall the notion of a normal sum \cite{Gompf1995}. Suppose $K_i$ is a knot in some closed $3$-manifold $Y_i$,  for $i=0,1$. The \dfn{normal sum of $Y_0$ and $Y_1$ along $K_0$ and $K_1$} is the result of removing neighborhoods of $K_i$ from $Y_i$ and identifying the resulting boundaries by a diffeomorphism that sends a longitude of $K_0$ to a longitude of $K_1$ (thus the result depends on a framing on the $K_i$) and the meridian of $K_0$ to the meridian of $K_1$. One can effect a normal sum by taking the connected sum of $Y_0$ and $Y_1$ on a neighborhood of a point on $K_0$ and $K_1$. Then in $Y_0\# Y_1$ we have the connected sum of $K_0$ and $K_1$. The normal sum is then the result of Dehn surgery on $K_0\# K_1$ with framing determined by the framings on the $K_i$.

\begin{proposition}\label{bound} 
A small Seifert fibered space $Y$ with complementary legs bounds a rational homology ball if and only if $Y$ is obtained by normal summing a lens space $L(p,q)$ (including the case of $S^3$) that bounds a rational homology ball and $S^1\times S^2$ along the core of a Heegaard torus in the lens space and the fiber of a Seifert fibration on $S^1\times S^2$. In particular, $Y$  will be given by the surgery diagram in Figure~\ref{fig:balanced}, for some $n\in\Z$, where  either $n_2=0$ or $p/q=[a^2_1,\ldots, a^2_{n_2}] \in \mathcal{R}$. \end{proposition}

Before setting the stage for the proof we observe that Theorem~\ref{lecuonas} is now an immediate consequence of Proposition~\ref{bound}.

The classification of small Seifert fibered spaces with complementary legs
that bound rational homology balls was established by Lecuona in
\cite[Corollary~3.3]{Lecuona2019}. The notation in \cite{Lecuona2019}
differs from the notation used here, and the statements there are made
only up to (possible) reversal of orientation, whereas our result applies
to either orientation. Accordingly, we briefly recall the classification
in \cite[Corollary~3.3]{Lecuona2019} and explain how it translates into
our conventions. In particular, \cite[Corollary~3.3]{Lecuona2019} asserts
that a small Seifert fibered space with complementary legs bounds a rational
homology ball if and only if it is homeomorphic, up to orientation, to
one of the manifolds described by the surgery diagram in
Figure~\ref{fig:lecuona}.

\begin{figure}[htb]{
\begin{overpic}
{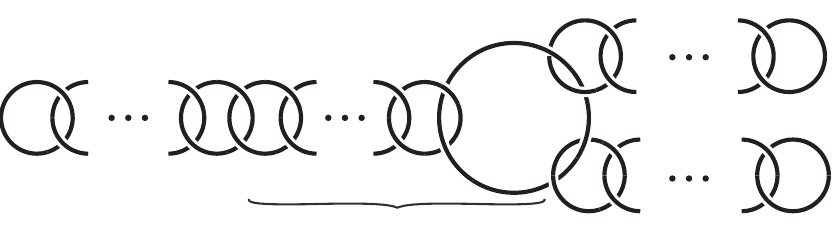}
\put(93, 23){$-a_t^2$}
\put(5, 23){$-a^2_{n_2}$}
\put(120, 77){$-2$}
\put(194, 77){$-2$}
\put(213, 89){$-2$}
\put(189, -1){$t$} 
\put(273, 105){$-a^1_1$}
\put(371, 105){$-a^1_{n_1}$}
\put(273, -5){$-a^3_1$}
\put(371, -5){$-a^3_{n_3}$}
\end{overpic}}
 \caption{Surgery diagram for a small Seifert fibered space with complementary legs that bounds a rational homology ball from Lecuona's paper \cite{Lecuona2019}. Here $a/b=[a^1_1,\ldots, a^1_{n_1}]>1$, and $a/(a-b)=[a^3_1,\ldots, a^3_{n_3}]$, where $0\leq t\leq n_2$, $a_t^2>2$, and  $[a_t^2-1, a_{t+1}^2,\ldots, a_{n_2}^2] \in \mathcal{R}$ when $n_2 >0$ or $t=n_2=0$ and $a_0^2>2$. }
  \label{fig:lecuona}
\end{figure}
 
We first observe that the small Seifert fibered spaces appearing in
Figure~\ref{fig:lecuona} all satisfy $e_0 \leq -2$. Recall that for a small Seifert fibered space $Y$, one has $e_0(Y)+e_0(-Y)=-3$. It follows that for  any such $Y$, either $Y$ or $-Y$ has $e_0\leq -2$ and moreover $Y$ bounds a rational homology ball if and only if $-Y$ does. Therefore,  when determining
whether a given small Seifert fibered space bounds a rational homology
ball, we may always choose the orientation so that $e_0 \leq -2$.  
Consequently, Figure~\ref{fig:lecuona} indeed accounts for all 
small Seifert fibered spaces with complementary legs, up to orientation, that bound rational homology balls.  However, since our
goal is to study contact structures on these spaces, it is essential to keep track of
orientations. For this reason, we will instead work with the
orientation-sensitive diagrams shown in Figure~\ref{fig:balanced}.

We will now see that when  $e_0\leq -2$, the surgery diagrams depicted in Figures~\ref{fig:balanced} and~\ref{fig:lecuona} are in fact equivalent. We begin with the case of $e_0=-2$, and so in Figure~\ref{fig:balanced}, we have $n\leq -1$. We will focus on the central unknot and the chain to its left, as the rest of the diagram is unchanged. The equivalence is shown in Figure~\ref{fig:firstequivalence}. 
\begin{figure}[htb]{
\begin{overpic}
{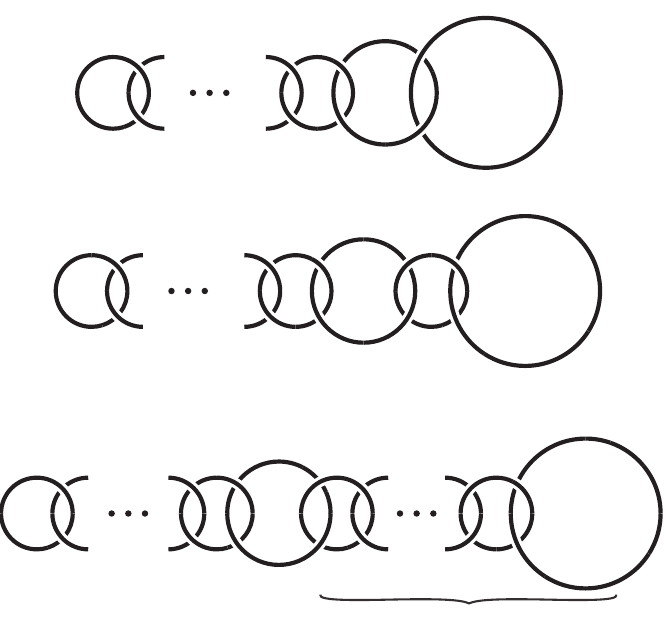}
\put(139, 278){$-a_1^2$}
\put(43, 278){$-a^2_{n_2}$}
\put(175, 284){$-n$}
\put(256, 285){$-1$}
\put(131, 182){$-a_1^2$}
\put(29, 182){$-a^2_{n_2}$}
\put(156, 190){$-n-1$}
\put(274, 190){$-2$}
\put(197, 182){$-1$}
\put(90, 75){$-a_1^2$}
\put(4, 75){$-a^2_{n_2}$}
\put(130, 82){$1$}
\put(301, 85){$-2$}
\put(153, 75){$-1$}
\put(175, 75){$-2$}
\put(228, 75){$-2$}
\put(216, -1){$-n$}
\end{overpic}}
 \caption{The equivalence between Figures~\ref{fig:balanced} and~\ref{fig:lecuona} when $e_0=-2$. Note that here $n \leq -1$. }
  \label{fig:firstequivalence}
\end{figure} 
The top row shows the chain mentioned above. A blow-up has been performed between the $-n$ and $-1$-framed components to obtain the second row. The last row is obtained by continued blow-ups between the $-1$-framed knot and what has become of the $(-n)$-framed knot under the blow-ups. Now one blows down the $1$-framed unknot to obtain the left-most chain in Figure~\ref{fig:lecuona} with $t=-n$, and $-a_t^2$ replaced with $-a_1^2-1$ and the other indices on the $a^2_j$ shifted. 
 
We now show  the equivalence between Figures~\ref{fig:balanced} and~\ref{fig:lecuona} when $e_0<-2$. In this case, we will have $n=0$ in Figure~\ref{fig:balanced}. Sliding the $-a_1^2$-framed unknot over the $-1$-framed unknot in Figure~\ref{fig:balanced} will result in Figure~\ref{fig:secondequivalence}.
\begin{figure}[htb]{
\begin{overpic}
{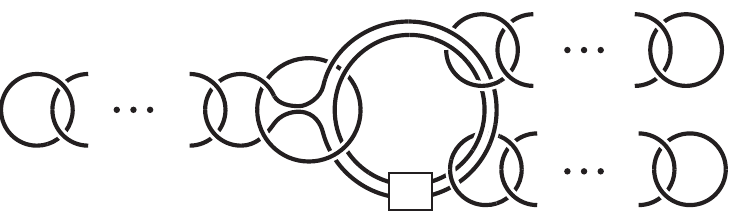}
\put(90, 76){$-a_1^2-1$}
\put(5, 76){$-a^2_{n_2}$}
\put(145, 82){$0$}
\put(189.5, 9){$-1$}
\put(182, 76){$-1$}
\put(220, 105){$-a^1_1$}
\put(319, 105){$-a^1_{n_1}$}
\put(220, -6){$-a^3_1$}
\put(319, -6){$-a^3_{n_3}$}
\end{overpic}}
 \caption{The equivalence between Figures~\ref{fig:balanced} and~\ref{fig:lecuona} when $e_0<-2$.}
  \label{fig:secondequivalence}
\end{figure}
One may now easily see that the $0$-framed unknot links the $-1$-framed unknot as a meridian and does not link any other component in the diagram. Thus, a sequence of handle slides will disentangle these two components from the rest of the diagram, and they can be deleted. This results in Figure~\ref{fig:lecuona} with $t=0$ and $-a_t^2$ replaced by $-a_1^2-1$ as desired. 

We note that after the above discussion, Lemma~\ref{bound} follows from \cite[Corollary~3.3]{Lecuona2019}, but we give a proof here as we will need the ideas in our discussion of symplectic rational homology balls below. We begin with a lemma that is equivalent to Proposition~3.1 in \cite{Lecuona2019}.
\begin{lemma}\label{qhc}  There is a rational homology cobordism   from the lens space $L(p,q)$ with $p/q=[a^2_1,\ldots, a^2_{n_2}]>1$ to the Seifert fibered space shown in Figure~\ref{fig:balanced} when $n_2 >0$  (and otherwise, when $n_2 =0$, the cobordism is from $S^3$).
\end{lemma}
Recall a cobordism $X$ from $Y_0$ to $Y_1$ is a rational homology cobordism if the relative homology $H_k(X,Y_i;\Q)=0$ for all $k$. (If this is true for one $i$ then it is true for the other $i$ as well.)
\begin{proof}
Given a $4$-manifold $X$ with boundary, we attach a \dfn{round $1$-handle} to $X$ as follows. The handle is $H=(D^1\times D^2)\times S^1$ and it is attached to $\partial X$ by an embedding $(\partial D^1)\times D^2\times S^1\to \partial X$. That is, it is attached along the neighborhood of two knots in $\partial X$ and the attaching map  is determined by framings on the knots. If $X'$ is the result of this attachment, then we note that $\partial X'$ is obtained from $\partial X$ by a normal sum along the attaching knots. We also note that a round $1$-handle attachment can be effected by attaching a $1$-handle to $\partial X$ so that the attaching sphere is two points, one on each knot, and then attaching a $2$-handle along the knot formed by the connected sum of the knots after the $1$-handle attachment. 

Assume that $p/q=[a^2_1,\ldots, a^2_{n_2}]>1$ for some $n_2 > 0$. Now consider $X=(L(p,q)\times [0,1])\cup (S^1\times D^3)$. We will attach a round $1$-handle to this disjoint union along the core of a Heegaard torus in $L(p,q)\times \{1\}$ and a regular fiber in a Seifert fibered structure on $S^1\times S^2$. This will give us a cobordism $X'$ from $L(p,q)\#S^1\times S^2$ to the Seifert fibered space shown in Figure~\ref{fig:balanced}. Indeed, after attaching the $1$-handle, we see the manifold in Figure~\ref{fig:balanced} without the $(-n)$-framed $2$-handle. The attaching knot of the  $(-n)$-framed $2$-handle is exactly the connected sum of a core of a Heegaard torus in $L(p,q)$ and a regular fiber in a Seifert fibration of $S^1\times S^2$. Thus, attaching the $(-n)$-framed $2$-handle completes the round $1$-handle attachment. 

Notice that $X'$ is also the result of attaching a $1$-handle to $L(p,q)\times [0,1]$ and then a $2$-handle that non-trivially runs over that $1$-handle. Hence $X'$ is clearly a rational homology cobordism. The same argument goes through when $n_2=0$ in Figure~\ref{fig:balanced} by replacing $L(p,q)$ with $S^3$. 
\end{proof}

\begin{proof}[Proof of Proposition~\ref{bound}]
If the small Seifert fibered space $Y$ with complementary legs is obtained from a lens space $L(p,q)$ (including the case of $S^3$) bounding a rational homology ball as stated in the proposition, then we can glue the cobordism built in Lemma~\ref{qhc} to $L(p,q)$ to obtain a rational homology ball with boundary $Y$. On the other hand, if $Y$ bounds a rational homology ball, then we may glue the cobordism from Lemma~\ref{qhc} (turned upside down) to this rational homology ball to build a rational homology ball with boundary $L(p,q)$. Finally, by Lisca~\cite{Lisca2007} the lens space $L(p,q)$ bounds a rational homology ball if and only if $p/q\in\mathcal R$, while for $n_2=0$ the relevant lens space is $S^3$, which bounds $B^4$; this yields the surgery description in Figure~\ref{fig:balanced}.
\end{proof}

\subsection{Contact structures on some small Seifert fibered spaces}\label{contactcomplementary}
We would now like to see various ways of understanding contact structures on small Seifert fibered spaces with complementary legs so that we can relate notation used in other papers with our perspective here. We start by discussing the standard tight contact structure on $S^1\times S^2$ and then discuss how to change contact surgery descriptions of contact structures on Seifert fibered spaces. We then discuss tight contact structures on small Seifert fibered spaces with a zero twisting Legendrian fiber. 

We begin by considering the standard contact structure on $S^1\times S^2$. One can think of $S^1\times S^2$ as $T^2\times [0,1]$ with curves of slope $\infty$ on $T^2\times\{i\}$, for $i=0,1$, collapsed to points. The contact structure is then $\ker (\cos(\pi t)\, d\theta + \sin(\pi t)\, d\phi)$. Notice that on $T^2\times \{i\}$ the characteristic foliation has slope $\infty$, and thus, we may perform a contact cut, \cite{Lerman01}, to get a contact structure on $S^1\times S^2$. This contact structure is the standard tight contact structure $\xi_{std}$. Notice that $T^2\times \{1/2\}$ is a Heegaard torus in $S^1\times S^2$, and it has a linear characteristic foliation of slope $0$. We can perturb this torus to be convex with two dividing curves, denote it by $T$. Then $T$ splits $S^1\times S^2$ into two solid tori $V_1$ and $V_3$, and each $V_i$ is a neighborhood of a Legendrian knot $L_i$. Moreover, one of the Legendrian dividing curves $L_2$ is also a Legendrian knot and Legendrian isotopic to both $L_1$ and $L_3$. (This is easily seen in the standard model of a Legendrian knot, \cite{EtnyreHonda01b}.) It is clear that the contact planes along each $L_i$ twist $0$ times with respect to the framing given by the product structure. The $L_i$ are shown in Figure~\ref{fig:s1s2surgery}. 
\begin{figure}[htb]{
\begin{overpic}
{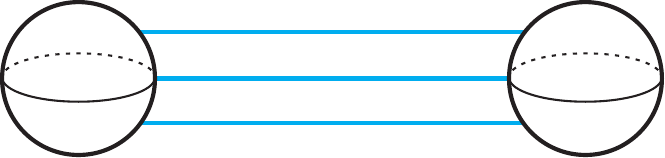}
\put(150, 22){$L_1$}
\put(150, 44){$L_2$}
\put(150, 66){$L_3$}
\end{overpic}}
 \caption{The Legendrian knots $L_1,L_2,$ and $L_3$ in $S^1\times S^2$.}
  \label{fig:s1s2surgery}
\end{figure}

Let $N_i$ be a standard neighborhood of $L_i$ and assume that all the $N_i$ are disjoint. Notice that $C_2=(S^1\times S^2)\setminus (N_1\cup N_3)$ is a thickened torus, $T^2\times I$, with an $I$-invariant contact structure on it. Thus, performing contact surgery on $L_1$ and $L_3$ is equivalent to taking the surgery tori and gluing them together along their boundary. Recall that the contact structure on a solid torus with meridional slope $m$ and dividing slope $d$ is determined by a minimal path in the Farey graph from $m$ clockwise to $d$ whose edges (except for the first) are decorated with a sign. As noted above $-a/b$ surgery on $L_1$ and $a/b$ surgery on $L_3$ result in $S^1\times S^2$. (We note that since the contact framing on $L_i$ is $0$, the smooth surgery coefficients and the contact surgery coefficients are the same.) We now recall Lemma~4.2 from \cite{EtnyreOzbagciTosun2025}. After \cite{EtnyreOzbagciTosun2025} was released the authors found out that the lemma had previously appeared in Matkovi\v c's work \cite{Matkovic23}.
\begin{lemma}\label{balanced}
Consider contact $(-a/b)$-surgery on $L_1$ and contact $(a/b)$-surgery on $L_3$. If the signs associated to the first surgery are the same and the signs for the second surgery are all opposite those of the first, then the resulting contact structure is $\xi_{std}$ on $S^1\times S^2$. Otherwise, the contact structure is overtwisted. 
\end{lemma}

We now move to tight contact structures on small Seifert fibered spaces with zero twisting fibers. It is not hard to show, see \cite{GhigginiLiscaStipsicz06, LiscaStipsicz07, Wu04}, that any tight contact structure on a small Seifert fibered space with zero twisting fiber can be obtained by contact surgery on the link in Figure~\ref{fig:s1s2surgery}. 
It is known \cite{Wu04, Wu06} that any small Seifert fibered space with $e_0\geq 0$ has a zero twisting fiber. Thus, all tight contact structures come from some contact surgery on the link in Figure~\ref{fig:s1s2surgery}. For $e_0>0$ we can arrange that $-1/r_2, -1/r_3\leq -2$ and $-1/r_1\in(-2,-1)$. Since these are all negative contact surgeries and the contact twisting along the link components is $0$, we know that all the resulting contact structures are tight (and Stein fillable). It was shown in \cite{GhigginiLiscaStipsicz06, Wu04} that all these contact structures are distinct. For $e_0=0$ we can arrange that all the $-1/r_i\leq -2$. Once again we see that all these contact structures are tight (and Stein fillable), but in \cite{GhigginiLiscaStipsicz06} we see that they are not all distinct. 

Moving to $e_0=-1$ Seifert fibered spaces, it is not true that all tight contact structures have zero twisting fibers, but the ones that do will again come from a contact surgery on the link in Figure~\ref{fig:s1s2surgery}. In this case, we can arrange that $-1/r_1,-1/r_3\leq -2$ and $-1/r_2>1$. Notice that in this case, we must do a positive contact surgery on $L_2$, and hence, we are not guaranteed that all such contact structures are tight. In fact, we will see below that some are not. Note that a small Seifert fibered space with complementary legs and $e_0=-1$ must have a zero twisting fiber, \cite{Wu06}. Thus, we can use the surgery diagram above to represent all tight contact structures on such a manifold. We note that this surgery diagram seems different from the one used in \cite{LiscaStipsicz07, Matkovic23} shown on the right of Figure~\ref{fig:LSdiagram}.
\begin{figure}[htb]{
\begin{overpic}
{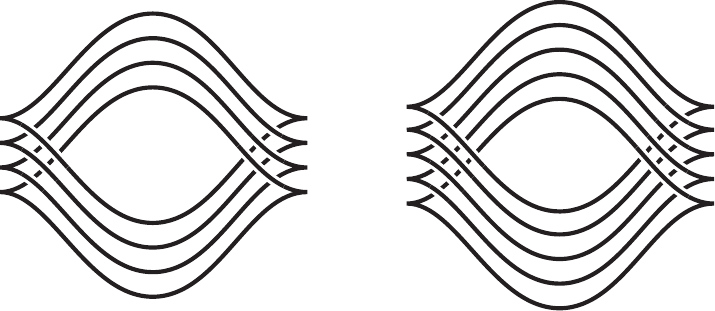}
\footnotesize
\put(151, 54){$(s_3)$}
\put(151, 67){$(s_2)$}
\put(151, 79){$(s_1)$}
\put(151, 91){$(+1)$}
\put(347, 49){$(-1/r_3)$}
\put(347, 61){$(-1/r_2)$}
\put(347, 73){$(-1/r_1)$}
\put(347, 85){$(+1)$}
\put(347, 96){$(+1)$}
\end{overpic}}
 \caption{The left diagram is equivalent to Figure~\ref{fig:s1s2surgery} with contact $(s_i)$-surgery performed on $L_i$. On the right is the diagram for $Y(-1;r_1,r_2,r_3)$.}
  \label{fig:LSdiagram}
\end{figure}
To see that these diagrams are really the same, we note that if we write $r_i=p_i/q_i$ then $Y(-1;r_1,r_2,r_3)$ will be given in Figure~\ref{fig:s1s2surgery} by performing $-1/r_i$ surgery on $L_i$ for $i=1,3$ and $q_2/(q_2-p_2)$ surgery on $L_2$. It is well known that the surgery diagram in Figure~\ref{fig:s1s2surgery} corresponds to the surgery diagram on the left in Figure~\ref{fig:LSdiagram}, see for example \cite{DingGeiges09}. In our case, the $s_i=-1/r_i$ for $i=1,3$, and $s_2=q_2/(q_2-p_2)$. Using the algorithm in \cite{DingGeigesStipsicz04} to convert a positive contact surgery into contact $(+1)$-surgeries and negative contact surgeries, we see that our diagram is equivalent to the one given on the right of Figure~\ref{fig:LSdiagram}.

We now focus on $Y=Y(e_0;r_1,r_2,r_3)$ with complementary legs and $e_0=-1$, and see how to go from the diagram on the right of Figure~\ref{fig:LSdiagram} to the diagram in Figure~\ref{fig:s1s2surgery} so that it fits into our discussion at the start of this section. Specifically, we will assume that $r_1+r_3=1$ so that we can write $r_1=p_1/q_1$ and $r_3=(q_1-p_1)/q_1$. In the $e_0=-1$ case, we see a contact surgery diagram for all possible tight contact structures on $Y$ given in the right-hand diagram of Figure~\ref{fig:LSdiagram}. Arguing as in the last paragraph, we can combine one of the contact $(+1)$-surgery curves and the curve with contact surgery coefficient $(-1/r_2)$  to obtain the diagram on the left with $s_1=-q_1/p_1\leq -2$, $s_2=-q_2/p_2\leq -2$, and $s_3=q_1/p_1$. Thus all contact structures on $Y$ come from surgery on the $L_i$ in Figure~\ref{fig:s1s2surgery} with contact surgery coefficients $(-1/r_1), (-1/r_2),$ and $(1/r_1)$. 

We will say contact $(r)$-surgery on $L$ is \dfn{consistent} if all of the signs determining the contact structure on the surgery torus have the same sign. We will say a contact surgery presentation for $Y$ in Figure~\ref{fig:s1s2surgery} is \dfn{balanced} if the contact surgery on $L_1$ and $L_3$ are both consistent and their signs are opposite. By Lemma~\ref{balanced} above it is clear that if $\xi$ comes from a balanced surgery diagram, then $\xi$ is Stein fillable (since $S^1\times S^2$ is). We have the following strengthening of this observation. But first, we recall from the introduction that $Y(-1;r_1,r_2,r_3)$ has \dfn{uniquely complementary legs} if $L_1$ and $L_3$ are complementary, and no initial subleg of $L_2$ is complementary with any initial subleg of $L_1$ or $L_3$. 

\begin{theorem}[Matkovi\v c 2023, \cite{Matkovic23}]\label{thmMatkovic}
A contact structure $\xi$ on $Y(-1;r_1,r_2,r_3)$ with uniquely complementary legs $L_1$ and $L_3$ is symplectically fillable if and only if its contact surgery presentation is balanced. 
\end{theorem}
We note that Matkovi\v c's theorem is much stronger than this and allows for complementary initial sublegs.

\section{Symplectic rational homology ball fillings when $e_0 \geq -1$} \label{eogeqm1}

Let $K$ be the $(-n)$-framed unknot in Figure~\ref{fig:balanced}. As discussed in the proof of Lemma~\ref{qhc}, the manifold described by Figure~\ref{fig:balanced} with $K$ removed is $L(p,q)\# (S^1\times S^2)$ and $K$ is the connected sum of the core of a Heegaard torus in $L(p,q)$ and a regular fiber in a Seifert fibration on $S^1\times S^2$.

\begin{lemma}\label{maxtb}
Let $\xi$ be the connected sum of any tight contact structure on $L(p,q)$  (including the case of $S^3$) and the tight contact structure on $S^1\times S^2$. Then there is a Legendrian representative of $K$ with contact twisting given by the blackboard framing in Figure~\ref{fig:balanced}.
\end{lemma}

\begin{proof}
One can find any torus knot in $S^1\times S^2$ as a leaf in a linear characteristic foliation of a Heegaard torus. See, for example, \cite[Lemma~4.2]{EtnyreOzbagciTosun2025}. So the contact framing of this leaf $K$ agrees with the framing coming from a Heegaard torus. As noted in the proof of Lemma~\ref{gencomplementarylegs} this framing agrees with the $0$-framing in Figure~\ref{fig:s1s2}. Now, any tight contact structure on $L(p,q)$ is obtained from some contact $(-p/q+1)$-surgery on the maximal Thurston-Bennequin unknot $U$. If $U'$ is a Legendrian push-off of $U$ (that is push $U$ slightly along the Reeb vector field), then $U'$ is a core of a Heegaard torus for $L(p,q)$. Clearly, $U'$ has contact twisting $-1$ (relative to the blackboard framing).  

We know that when connecting summing two Legendrian knots, the contact twisting of the connected sum is the sum of the contact twisting of the components plus $1$; see \cite{EtnyreHonda03}. Thus, the connected sum $K\# U'$ has contact framing agreeing with the blackboard framing in Figure~\ref{fig:balanced}.
\end{proof}

\begin{proposition}\label{SteinCobord}
There is a rational homology Stein cobordism from any tight contact structure on the lens space $L(p,q)$  (including the case of $S^3$) to a contact structure on the small Seifert fibered space whose smooth surgery diagram is depicted in Figure~\ref{fig:balanced} with $n\geq 1$. 
\end{proposition}
\begin{proof}
The rational homology cobordism constructed in Lemma~\ref{qhc} is built from $X=(L(p,q)\times [0,1])\cup (S^1\times D^3)$ by attaching a round $1$-handle. There is a Stein structure on $S^1\times D^3$ and we can use a part of the symplectication to put a symplectic structure on $(L(p,q)\times [0,1])$ with $L(p,q)\times \{0\}$ concave and $L(p,q)\times \{1\}$ convex. The round $1$-handle is now attached by attaching a $1$-handle connecting the components of $X$, followed by a $2$-handle attachment. A Stein structure can always be extended over a $1$-handle, and it can be extended over a $2$-handle if it is attached to a Legendrian knot and the framing of the handle is less than the contact framing (by stabilizing the Legendrian if necessary). Given the contact framing on a Legendrian realization of $K$ found in Lemma~\ref{maxtb}, we see the cobordism is Stein. 
\end{proof}

\begin{theorem}\label{maine00}
A small Seifert fibered space $Y$ with uniquely complementary legs given by $-a/b$ and $-a/(a-b)$ surgery, and $e_0=-1$ carries a contact structure which is symplectically fillable by a rational homology ball if and only if it is given by the surgery diagram in Figure~\ref{fig:balanced} where $[a^2_1,\ldots, a^2_{n_2}]=\frac {m^2}{mh-1}$ for some relatively prime integers $0<h<m$ when $n_2 >0$, and otherwise $n_2=0$, in which case we set $m=1$ and $h=0$, and $n\geq 2$.

Moreover, there are at least $\min\{\lfloor a/2\rfloor,n\}$, respectively $\min\{\lfloor a/2\rfloor,2n\}$, pairwise nonisotopic tight contact structures with this property on such a small Seifert fibered space if $m=1$, respectively $m>1$. If $2n$, respectively $n$, is less than or equal to $a/2$, then there are exactly $2n$, respectively $n$, such contact structures when $m>1$, respectively $m=1$. Furthermore, if $n>2$, then for $2n-2$, respectively $n-2$, of these contact structures, there is a unique symplectic rational homology ball filling of the contact manifold, up to blowups.
\end{theorem}
We note that Theorem~\ref{thm: minusone} is an immediate corollary of this result. 

\begin{proof}
Lisca \cite{Lisca08} has shown that the universally tight contact structure $\xi_{can}$ on $L(m^2, mh-1)$ is the convex boundary of a symplectic rational homology ball, and \cite{ChristianLi23, EtnyreRoy21, Lisca08} show that these are the only contact structures on any lens spaces admitting symplectic rational homology ball fillings. (The case when $m=1$ and $h=0$ corresponds to using the standard tight contact structure on $S^3$ in the construction.) Thus, using the Stein cobordism from Proposition~\ref{SteinCobord} we see that the claimed $Y$ admits rational homology ball fillings.  

To see the other implication, we first note that if the contact structure on $Y$ is not balanced, then it is not symplectically fillable by Theorem~\ref{thmMatkovic}. Thus, we can assume the surgery description of $Y$ is balanced. Recall in Section~\ref{contactcomplementary} (before Theorem~\ref{thmMatkovic}) we saw that all contact structures on $Y$ are obtained from contact surgery on the link in Figure~\ref{fig:s1s2surgery}. Moreover, for contact structures that might bound symplectic rational homology balls the surgeries on $L_1$ and $L_3$ are balanced, and hence we see that all these contact structures come from negative contact surgery on $L_2$, which after the surgeries on $L_1$ and $L_3$ is a torus knot in $S^1\times S^2$ (that is $L_2$ is the fiber $F$ in Figure~\ref{fig:s1s2}). All these surgeries can be effected by attaching a Stein round $1$-handle to $L_2$ and the core of a Heegaard torus in a lens space. Thus, for any contact structure on $Y$ we can build a Stein cobordism (as in the proof of Proposition~\ref{SteinCobord}) from the corresponding lens space to $Y$. Since this cobordism is a rational homology cobordism, it is clear that the $\theta$-invariant of the contact structure on $Y$ is the same as the $\theta$-invariant of the contact structure on the lens space. For $Y$ to smoothly bound a rational homology ball the lens space must be of the form $L(p,q)$ where $p/q\in \mathcal{R}$, \cite{Lisca2007}. In Lemma~9.4 (and Remark~1.8) of \cite{EtnyreOzbagciTosun2025} we showed that any tight contact structure $\xi$ on $L(p,q)$, where $p/q\in \mathcal{R}$, satisfies $\theta (\xi)>-2$ unless $p/q=m^2/(mh-1)$ and $\xi$ is isotopic to $\pm\xi_{can}$. Thus, if $Y$ has a rational homology ball filling it must be of the form claimed. 

From the discussion above, any contact structure on $Y$ that has a symplectic rational homology ball filling is the upper boundary of the symplectic cobordism built from a portion of the symplectization of $(L(m^2,mh-1), \pm \xi_{can})$ union $S^1\times D^3$ by attaching a round $1$-handle to $L_2$ in $S^1\times S^2$ and the core of a Heegaard torus in the lens space. After attaching the $1$-handle in the round $1$-handle we have upper boundary the lens space connect sum $S^1\times S^2$ and the $2$-handle of the round $1$-handle is then attached to the connected sum of a fiber in a Seifert fibration on $S^1\times S^2$ and the core of a Heegaard torus for the lens space. As argued in the proof of Lemma~\ref{maxtb} the Legendrian has contact twisting $0$. Since we are attaching it with framing $-n\leq -2$, we see we must stabilize the Legendrian knot $n-1$ times. Thus, there are $n$ possible ways to attach this $2$-handle. Hence, there are $2n$ possible Legendrian surgery diagrams for $Y$ if $m\not=1$ and $n$ if $m=1$ (the $2$ comes from the diagrams for $\pm \xi_{can}$ on the lens space and the $n$ comes from the choices for the $2$-handles). These surgery diagrams might not produce distinct contact structures. To address this question, we start with the $m=1$ case. In this case, the $2$-handle is attached to a torus knot in $S^1\times S^2$ that runs over the $1$-handles $a$ times (see \cite[Lemma~4.1]{EtnyreOzbagciTosun2025}). Thus, one can see that the second cohomology of the Stein filling is $\Z/a\Z$. The Chern class of the Stein manifold can assume at most $a$ values. As all the rotation numbers for the possible realizations for the $(-n)$-framed $2$-handle differ by an even number, one may easily see that if $n>a$ we can realize $a$ distinct Chern classes in the surgery diagrams (see \cite[Theorem~4.12]{Gompf98} to see how the rotation numbers contribute to the Chern class) if $a$ is odd and $a/2$ if $a$ is even. Thus, by \cite[Theorem~1.2]{LiscaMatic97} we know the contact structures associated to Stein fillings with distinct Chern classes must be distinct. This gives our bound in the case when $m=1$. When $m>1$, there are two possible choices for the Chern class of the lens space, and this leads to $2n$ possible Chern classes, but as above, at most $a$ or $a/2$ can be distinguished by the Chern class, depending on whether $a$ is odd or even. 

Now, when $2n$, respectively $n$, is less than $a/2$, we see that there are exactly $2n$, respectively $n$, contact structures with rational homology ball fillings. This follows because the second paragraph of this proof shows how all such contact structures are constructed, and above we saw that all of these will be distinct as the Chern classes on the symplectic manifolds they bound are distinct. 

Finally, using \cite[Theorem~1.4]{ChristianMenke19pre} we see that if the $2$-handle is attached along a Legendrian knot that has been stabilized both positively and negatively, then any symplectic filling will come from a symplectic filling of the connected sum of the lens space and $S^1\times S^2$ with this $2$-handle attached. Thus, when $m=1$, we see that $n-2$ of the above contact structures will have only one symplectic rational homology ball filling. 

Similarly, \cite[Theorem~1.4]{ChristianMenke19pre} shows that if the stabilizations of the Legendrian knot are all of one sign and those defining the contact structure on the lens space are of the opposite sign, then the filling will be unique ({\em cf } \cite[Theorem~3.1]{EtnyreRoy21}). Thus, each of those $2n-2$ tight structures will have only one symplectic rational homology ball filling. 
\end{proof}

\begin{theorem}\label{maine0+}
Let $Y$ be a small Seifert fibered space with complementary legs and $e_0\geq 0$. A balanced contact structure $\xi$ on $Y$ bounds a symplectic rational homology ball if and only if $Y$ is given by the surgery diagram in Figure~\ref{fig:balanced}, where $[a^2_1,\ldots, a^2_{n_2}]=\frac {m^2}{mh-1}$ for some relatively prime integers $0<h<m$, and $n=1$. Moreover, on such a $Y$, there are four balanced contact structures (up to isotopy) that bound a symplectic rational homology ball.

If $\xi$ is not a balanced contact structure, then $(Y, \xi)$ might bound a symplectic rational homology ball only if $Y$ is given by the surgery diagram in Figure~\ref{fig:balanced} where either $n_2=0$ or $[a^2_1,\ldots, a^2_{n_2}]\in \mathcal{R}$. Given such a $Y$ there are at most two, respectively four, contact structures (up to isotopy) that can bound a symplectic rational homology ball if $e_0=0$, respectively $e_0>0$.
\end{theorem}

\begin{proof}

The proof for balanced contact structures is identical to the proof in the
$e_0=-1$ case given in Theorem~\ref{maine00}; note that the hypothesis $0<h<m$
excludes the degenerate case of $S^3$, so the lens space contributes
both structures $\pm\xi_{can}$ and the count of four is uniform. One must only be careful that the surgery diagrams yield distinct contact structures, but this is always true for $e_0>0$ and also for $e_0=0$ when the legs are balanced, see the proof of Theorem~2.4 in \cite{GhigginiLiscaStipsicz06}. Thus, there are $4$ tight contact structures that bound rational homology balls; they come from the $2$ choices for the stabilizations on the complementary legs and the $2$ choices from the possible contact structures on the lens spaces that bound rational homology balls. 

If $\xi$ is not balanced, then we recall from the proof of Proposition~5.1 in \cite{EtnyreOzbagciTosun2025} that if the contact structure $\xi$ on $Y$ admits a rational homology ball filling, it must have consistent legs. Since we are now considering non-balanced legs, there are $4$ possible contact structures since the complementary legs all have one sign and the other leg can have any sign. This gives $4$ distinct contact structures when $e_0>0$. This, coupled with Lecuona's work cited in Theorem~\ref{lecuonas}, finishes the case when $e_0>0$. 

We now consider the case that $e_0=0$. In \cite{GhigginiLiscaStipsicz06} it was shown that for $e_0=0$, some contact surgeries realizing $Y(0;r_1,r_2, r_3)$ yield the same contact manifold. Specifically, if $-1/r_i<-2$ for any $i$ and all the legs have the same sign, then one can change a consistent leg into an inconsistent leg (see the proof of Theorem~2.4 in \cite{GhigginiLiscaStipsicz06}). Thus, these contact structures do not admit a symplectic rational homology ball filling, as noted above. If $-1/r_i=-2$ for all $i$, then we can rule out the existence of a symplectic rational ball filling altogether. Indeed, by Proposition~\ref{bound}, $Y(0;\tfrac12,\tfrac12,\tfrac12)$ does not even
bound a smooth rational homology ball, since its associated fraction $p/q=3\notin\mathcal{R}$.  Thus there are at most $2$ contact structures with symplectic ball filling verifying the theorem in the case that $e_0=0$. 
\end{proof}

We finally note that Theorem~\ref{thm: minusone} and Theorem~\ref{thm: nonneg} are now a rephrasing of Theorem~\ref{maine00} and Theorem~\ref{maine0+}, respectively. Namely by computing the surgery coefficient associated to $[n,a^2_1,\ldots, a^2_{n_2}]$ where $[a^2_1,\ldots, a^2_{n_2}]=\frac{m^2}{mh-1}$ yields Theorem~\ref{thm: minusone} when $n\geq 2$ and Theorem~\ref{thm: nonneg} when $n=1$. Here the case $n_2=0$ corresponds to $m=1$, $h=0$, where the bracket is empty and the coefficient reduces to the integer $-n$ (the $S^3$ case); it contributes to Theorem~\ref{thm: minusone} when $n\geq 2$ and is excluded from the balanced statement of Theorem~\ref{thm: nonneg} by the hypothesis $0<h<m$.

\section{Symplectic rational homology ball fillings when $e_0 \leq -2$}\label{m2case}
This section is devoted to the proof of Theorem~\ref{thm: smallSFSnotfillable} which says that a small Seifert fibered space $Y(e_0; r_1, r_2, r_3)$ with complementary legs and $e_0\leq -2$ does not admit a rational homology ball symplectic filling. 

Recall we are assuming that the complementary legs are defined by  $r_1+r_3=1$.  Set $1/r_k=[a_1^k,\ldots, a_{n_k}^k]$, $a^2_0=-e_0$, and $p/q=[a^2_0,a^2_1,\ldots, a^2_{n_2}]$. We would like to warn the reader that in this section and the next one,  the symbol $p/q$ is redefined as $p/q=[a^2_0,\dots,a^2_{n_2}]$ with $a^2_0=-e_0$, versus
$p/q=[a^2_1,\dots,a^2_{n_2}]$ as in the previous two sections. We notice that the $a_j^3$ terms are determined by the $a_i^1$ terms by the Riemenschneider point rule \cite{Riemenschneider1974}, and thus  $Y(e_0; r_1, r_2, r_3)$ is completely determined by the strings ${\bf a^1}=(a_1^1,\ldots, a_{n_1}^1)$ and ${\bf{a^2}}=(a_0^2, a_1^2,\ldots, a_{n_2}^2)$. Hence  we denote $Y(e_0; r_1, r_2, r_3)$ by $M_{{\bf a^1}, {\bf a^2}}$. If we set $p/q=[a_0^2, a_1^2,\ldots, a_{n_2}^2]$ then Lecuona's result discussed in Theorem~\ref{lecuonas} says that $M_{{\bf a^1}, {\bf a^2}}$ bounds a rational homology ball if and only if $(p-q)/q'\in \mathcal{R}$ where $0<q'<p-q$ is the reduction of $q$ modulo $p-q$. The equivalence of this criterion with the one given in the caption of Figure~\ref{fig:lecuona} can be seen by Lemma~\ref{lem: equiv}. 

\begin{lemma}\label{lem: equiv} Let $p/q=[a_0^2, a_1^2,\ldots, a_{n_2}^2]$. Suppose that  there is some  $ 1 \leq t \leq n_2$ such that $a^2_i =2$ 
for $ 0 \leq i \leq t-1$ and $a^2_t > 2$.  If we set 
$p_t/q_t=[a_t^2, a_{t+1}^2,\ldots, a_{n_2}^2]$, then we have 
\[
\frac{p-q}{q'}= \frac{p_t-q_t}{q_t}= [a_t^2-1, a_{t+1}^2,\ldots, a_{n_2}^2],
\] 
where $0<q'<p-q$ is the reduction of $q$ modulo $p-q$.  
\end{lemma} 

\begin{proof} Under the given assumptions, we have 
\[
\frac{p}{q}=[a_0^2, a_1^2,\ldots, a_{n_2}^2]=\frac{(t+1)p_t -tq_t}{tp_t - (t-1)q_t}
\] 
and hence 
\[
\frac{p-q}{q}=\frac{p_t-q_t}{{tp_t - (t-1)q_t}}
\]
which implies that
\[
\frac{p-q}{q'}=\frac{p_t-q_t}{q_t},
\]
since $q'=q_t$. This last term is easily seen to be $[a_t^2-1, a_{t+1}^2,\ldots, a_{n_2}^2]$.
\end{proof}

We now consider the case when $e_0\leq -2$. In this case, the plumbing diagram defining $M_{{\bf a^1}, {\bf a^2}}$ is negative definite. Thus, it is the {\em oriented} link of a normal complex surface singularity, and therefore it has a canonical contact structure $\xi_{can}$ (also known as the Milnor fillable contact structure), which is unique up to contactomorphism \cite{CaubelNemethiPopescu-Pampu2006}.   

We now recall Proposition~1.7 from \cite{EtnyreOzbagciTosun2025}.
\begin{proposition}\label{prop: rational}  If $\xi$ is any tight contact structure on the small Seifert fibered space $M_{{\bf a^1}, {\bf a^2}}$, which is not isotopic to $\pm \xi_{can}$,   then we have $ \theta(\xi_{can}) < \theta(\xi).$
\end{proposition}
We also recall that $\theta(-\xi)=\theta(\xi)$, see \cite[Theorem~4.5]{Gompf98}, so the above proposition addresses all contact structures on $M_{{\bf a^1}, {\bf a^2}}$.
Our main theorem in the $e_0\leq -2$ case will follow from the following computation. 
\begin{proposition} \label{prop: minSFS} The canonical contact structure $\xi_{can}$ on $M_{{\bf a^1}, {\bf a^2}}$,  satisfies
$-2 < \theta (\xi_{can})$, provided that  $(p-q)/q' \in\mathcal{R}$, where $0 < q' < p - q$ is
the reduction of $q$ modulo $p - q$.
 \end{proposition}
 
Based on Lecuona's characterization, the proof of Theorem~\ref{thm: smallSFSnotfillable} is obtained by combining Proposition~\ref{prop: rational}  and Proposition~\ref{prop: minSFS} as follows.

\begin{proof} [Proof of Theorem~\ref{thm: smallSFSnotfillable}] Let $e_0 \leq -2$. Suppose that the small Seifert fibered space $Y(e_0; r_1, r_2, r_3)$ has complementary legs and assume without loss of generality that $r_1+r_3=1$, so that $Y(e_0; r_1, r_2, r_3)= M_{{\bf a^1}, {\bf a^2}}$ with two complementary legs, and ${\bf a_1}$ and ${\bf a_2}$ as defined above. Lecuona \cite{Lecuona2019} showed that $M_{{\bf a^1}, {\bf a^2}}$ smoothly bounds a rational homology ball if and only if $(p-q)/q' \in\mathcal{R}$. On the other hand, Proposition~\ref{prop: minSFS} together with Proposition~\ref{prop: rational} imply that $-2 < \theta (\xi)$ for any tight contact structure $\xi$ on  $M_{{\bf a^1}, {\bf a^2}}$, provided that  $(p-q)/q' \in\mathcal{R}$.  Since the $\theta$-invariant of any tight contact structure that admits a rational homology ball symplectic filling is equal to $-2$, we immediately conclude that no rational homology ball with boundary $M_{{\bf a^1}, {\bf a^2}}$ can symplectically fill any contact structure on $M_{{\bf a^1}, {\bf a^2}}$.\end{proof}

In order to compute the $\theta$-invariant we recall a useful definition from \cite{Lisca2007}. 
Suppose that $r/s=[a_0, a_1, \ldots, a_k]$, where $a_i \geq 2$ for all $0 \leq i \leq k$. We set $$I(r/s)=\sum_{i=0}^{k} (a_i-3).$$ 
The proof of Proposition~\ref{prop: minSFS} is based on the following result, which is the main technical part of the paper. 

\begin{proposition} \label{prop: thetaSFS}   Using the notation above we set  $\widetilde{p}/\widetilde{q}=[a^1_1, a^1_2, \ldots, a^1_{n_1}]$. We have the following formula 
\begin{equation}\label{eq: thetaSFS}
\theta (\xi_{can}) =1- I(p/q)-\dfrac{1}{[a^2_{n_2}, a^2_{n_2-1}, \ldots, a^2_1, a^2_0-1]} +\dfrac{2(\widetilde{p}-2)}{\widetilde{p}(p-q)} -\dfrac{(\widetilde{p}-2)^2 q}{(\widetilde{p})^2(p-q)} \end{equation}
for the canonical contact structure $\xi_{can}$ on $M_{{\bf a^1}, {\bf a^2}}$. 
\end{proposition} 

\begin{remark} Note that  Proposition~\ref{prop: thetaSFS} is a generalization of \cite[Proposition 9.8]{EtnyreOzbagciTosun2025}, since the spherical $3$-manifold $D(p,q)$ is a small Seifert fibered space with two complementary legs with ${\bf a_1}={\bf a_3}=(2)$, which implies that $\widetilde{p}=2$, and therefore the sum of the last two terms in the Formula~(\ref{eq: thetaSFS}) vanishes for the case of $D(p,q)$. 
\end{remark}

We will prove this proposition in the next section and now give the proof of Proposition~\ref{prop: minSFS}.

\begin{proof}[Proof of Proposition~\ref{prop: minSFS}]
We consider two cases.

\smallskip
\noindent {{\bf Case (i):}} Suppose that $p-q >q$, which is equivalent to assuming that $-e_0=a_0^2 >2$.  It follows that $q'=q$ and  the hypothesis $(p-q)/q \in\mathcal{R}$ implies that $I((p-q)/q) \leq 1$ by Lisca's work \cite{Lisca2007} (see also, \cite[Appendix~A.2]{AcetoGollaLarsonLecuona2020pre}, where the $I$-values of all rational numbers in $\mathcal{R}$ are listed explicitly with maximum value of $1$). One may easily compute that $I((p-q)/q) = I(p/q) -1$. 

If we assume that $I((p-q)/q) \leq 0$ (and thus $I(p/q) \leq  1$), then the inequality  $-2 < \theta (\xi_{can})$ follows immediately from Formula~\eqref{eq: thetaSFS}. This is because the term $\dfrac{2(\widetilde{p}-2)}{\widetilde{p}(p-q)}$ in Formula~(\ref{eq: thetaSFS}) is always nonnegative, the term  $-\dfrac{(\widetilde{p}-2)^2 q}{(\widetilde{p})^2(p-q)}$ is greater than $-1$ since we assumed that $p-q >q$, and the term $-\dfrac{1}{[a^2_{n_2}, a^2_{n_2-1}, \ldots, a^2_1, a^2_0-1]}$ is strictly greater than $-1$ since $a^2_0 > 2$.

Now suppose that $I((p-q)/q)=1$. Since we assumed that  $(p-q)/q \in\mathcal{R}$, it follows that $(p-q)/q$ belongs to class (V) in  \cite[Appendix~A.2]{AcetoGollaLarsonLecuona2020pre} and the continued fraction expansion of $(p-q)/q$ is obtained from $[4]$ by final-$2$ expansions. An equivalent characterization of such a rational number $(p-q)/q$  is that $p-q=m^2$ and $q=mh-1$  for some coprime integers $0< h < m$ (see, for example, Proposition 4.1 in  \cite{StipsiczSzaboWahl08}). 

Recall that $p/q=[a_0^2, a_1^2,\ldots, a_{n_2}^2]$ and we assumed that $a_0^2 >2$. It immediately follows that $[a_0^2-1, a_1^2,\ldots, a_{n_2}^2]=(p-q)/q$ and by a well-known fact about continued fraction expansions we have $[a^2_{n_2}, a^2_{n_2-1}, \ldots, a^2_1, a^2_0-1]=(p-q)/q^*$, where  $0 < q^* < p-q$ is the multiplicative inverse of $q$ mod $(p-q)$, see \cite[Lemma~A4]{OrlikWagreich77}. Therefore, we have 
\[
\dfrac{1}{[a^2_{n_2}, a^2_{n_2-1}, \ldots, a^2_1, a^2_0-1]}= \dfrac{q^*}{p-q}.
\]   Since $q^*=m(m-h)-1$, we have $2+q+q^*=p-q$. It follows that 
\[
-\dfrac{1}{[a^2_{n_2}, a^2_{n_2-1}, \ldots, a^2_1, a^2_0-1]}  -\dfrac{(\widetilde{p}-2)^2 q}{(\widetilde{p})^2(p-q)}
> -\dfrac{q^*}{p-q} -\dfrac{q}{p-q}= -1.
\]

Since we assumed that $I((p-q)/q)=1$ (and therefore $I(p/q)=2$), we have $1- I(p/q) = -1$. Taking into account the fact that $\dfrac{2(\widetilde{p}-2)}{\widetilde{p}(p-q)}$ is always nonnegative, we again conclude by Formula~(\ref{eq: thetaSFS}) that $-2 < \theta (\xi_{can})$. 

\smallskip

\noindent {{\bf Case (ii):}} Suppose that $p-q <q$, which is equivalent to assuming that $-e_0=a^2_0 = 2$. If  $a^2_i =2$, for all $0 \leq i \leq n_2$, then $p=n_2+2$, $q=n_2+1$ and $q^*=n_2+1$, which implies that 
\[
\theta (\xi_{can}) =1+(n_2+1)-1+\dfrac{2(\widetilde{p}-2)}{\widetilde{p}} -\dfrac{(\widetilde{p}-2)^2 }{(\widetilde{p})^2} (n_2+1)> 0.
\]
 Otherwise,  there exists $0 \leq t < n_2$ such that $a^2_i =2$, for all $0 \leq i \leq t$, but $a^2_{t+1} > 2$. In other words,  we have 
 \[
 p/q = [\underbrace{2, \ldots 2}_{t+1}, a^2_{t+1},  \ldots, a^2_{n_2}] \;\; \mbox{and} \;\; b_{t+1} > 2.
 \]  
 By setting $p_t/q_t=[a^2_{t+1},  \ldots, a^2_{n_2}]$, one may compute that $p_t -q_t =p-q$ and $ q_t=q'$. Hence, the condition  $(p-q)/q' \in\mathcal{R}$ is equivalent to the condition  $(p_t -q_t) /q_t \in \mathcal{R}$, and combined with the observation that $p_t-q_t > q_t$,  {\bf Case (i)} of our proof  implies that 
 \begin{equation} \label{eq: ineq} 
 1- I(p_t/q_t)-\dfrac{1}{[a^2_{n_2}, a^2_{n_2-1}, \ldots, a^2_{t+2}, a^2_{t+1}-1]} +\dfrac{2(\widetilde{p}-2)}{\widetilde{p}(p_t-q_t)} -\dfrac{(\widetilde{p}-2)^2 q_t}{(\widetilde{p})^2(p_t-q_t)} > -2,
 \end{equation} 
  but note that the $\widetilde{p}$ and $\widetilde{q}$ are different in this equations (as their role below is unimportant, we do not establish new notation for them here). 

  Next, we observe that 
\begin{enumerate}[label=(\roman*)]
  \item $I(p_t/q_t)=I(p/q)+(t +1)$,
 \item $[a^2_{n_2}, a^2_{n_2-1}, \ldots, a^2_{t+2}, a^2_{t+1}-1] = [a^2_{n_2}, a^2_{n_2-1}, \ldots, a^2_1, a^2_0-1]$, and
  \item $q_t/(p_t-q_t) =q/(p-q)-(t+1)$,
\end{enumerate} 
where the last item is proven via induction on $t$.
By substituting these in~(\ref{eq: ineq}), we deduce  
 $$\theta (\xi_{can}) =1- I(p/q)-\dfrac{1}{[a^2_{n_2}, a^2_{n_2-1}, \ldots, a^2_1, a^2_0-1]} +\dfrac{2(\widetilde{p}-2)}{\widetilde{p}(p-q)} -\dfrac{(\widetilde{p}-2)^2 q}{(\widetilde{p})^2(p-q)} > -2$$ using the fact that $ 0 \leq (\widetilde{p}-2)/\widetilde{p} < 1$. 
\end{proof}

\section{Computing the $\theta$-invariant}   \label{subsec: theta} 

The goal of this subsection is to describe a proof of Proposition~\ref{prop: thetaSFS}, which is the main technical part of the paper. We begin with the following definition.

\begin{definition} \label{def: det} For any integer $n>0$ and any sequence of integers $m_1, m_2, \ldots, m_n$ with $m_i \geq 2$, we set $$M=M(m_1, m_2, \ldots, m_n)= \left[\begin{array}{cccc}
 -m_1 & 1 & &   \\
1 & -m_2 & \ddots & \\
 & \ddots & \ddots & 1 \\
 & & 1 & -m_n 
\end{array}\right].$$  For any $1 \leq i \leq n$, we define $u^M_i$ as the absolute value of the determinant of the top-left $i \times i$ submatrix of $M$, and we set  $u^M_0=1$.  Let ${\bf u}^M=[u^M_0 \; u^M_1 \; \cdots \; u^M_{n-1}]^T \in \mathbb{R}^n$. Similarly, for any $1 \leq j \leq n$, we define $v^M_j$ as the absolute value of the determinant of the bottom-right $j \times j$ submatrix of $M$, and we set $v^M_0=1$. Let ${\bf v}^M=[v^M_{n-1} \; v^M_{n-2} \; \cdots \; v^M_0]^T \in \mathbb{R}^n$. 
\end{definition}

\begin{lemma} \label{lem: firstrow} 
Using the notation as in Definition~\ref{def: det}, let ${s}/{t}=[m_1, m_2, \ldots, m_n]$ for some relatively prime integers $0 < t < s$. Then $\det M = (-1)^n s$,   the first column of $M^{-1}$ is given by the vector  $-\dfrac{1}{s}   \; {\bf v}^M$ and its dot product with the vector $[m_1-2 \; m_2-2 \; \cdots \; m_n-2]^T \in  \mathbb{R}^n$ is equal to $-1+ \dfrac{1+t}{s}$. Moreover,  the last column of $M^{-1}$ is given by the vector  $-\dfrac{1}{s}   \; {\bf u}^M$. Furthermore, the first entry of  ${\bf v}^M$ is $t$. 
\end{lemma} 

\begin{proof} We first show that $\det M = (-1)^n s$. To prove this claim, we set 
\[
d_i = \det M(m_1, m_2, \dots, m_i) \; \mbox{and} \; \frac{s_i}{t_i} = [m_1, m_2, \dots, m_i] \; \mbox{for}\; 1 \leq i \leq n.
\]
We also set  $d_0=1=s_0$ and $t_0=0$. Note that, by definition, $d_1=-m_1$, and $d_n=\det M$. Computing $d_i$ using the cofactor expansion about the last row shows that the $d_i$'s satisfy the  following recursive relation 
\[
d_i=-m_id_{i-1}-d_{i-2} \; \mbox{for}\; 2 \leq i \leq n.
\] 
On the other hand, it is well-known and easily proven by induction, that 
\[
s_i =m_i s_{i-1}-s_{i-2} \; \mbox{and}\; t_i =m_i t_{i-1}-t_{i-2} 
\] 
hold for $ 2 \leq i \leq n$. Note that, by definition, $s_1=m_1, t_1=1, s_n=s$, $t_n=t$. Comparing the recursive relations for $d_i$'s and $s_i$'s we see that $d_i = (-1)^i s_i$ for $ 0 \leq i \leq n$. In particular, for $i=n$ we have 
$$
\det M=d_n= (-1)^n s_n=  (-1)^n s,
$$ 
as claimed in the lemma.  

Next, we observe that for any $1 \leq j \leq n$,  the determinant of the submatrix $M_{1,j}$ of $M$ obtained by deleting its first row and $j$th column is equal to the determinant of the bottom-right $(n-j) \times (n-j)$ submatrix of $M$, simply expanding along the first columns as we compute the determinants. Since $v^M_{n-j}$ is defined as the {\em absolute value} of the determinant of the bottom-right $(n-j) \times (n-j)$ submatrix of $M$, we see that $\det M_{1,j}=(-1)^{n-j} v^M_{n-j}.$ It follows that the  
  $(1, j)$-cofactor of $M$ is $$(-1)^{1+j} \det M_{1,j}= (-1)^{1+j} (-1)^{n-j}v^M_{n-j}=(-1)^{n+1} v^M_{n-j}.$$ 
  This implies that  the $(j,1)$-entry of $M^{-1}$ is equal to $$\frac{(-1)^{1+j} \det M_{1,j}} {\det M} = \frac{(-1)^{n+1} v^M_{n-j}}{(-1)^n s}=-\frac{v^M_{n-j}}{s}$$  for any $1 \leq j \leq n$. Therefore, the first column of $M^{-1}$ is given by the vector  $-\dfrac{1}{s}   \; {\bf v}^M$. 
  
Since the first row of $M^{-1}$ is equal to the transpose of its first column,  the dot product of $-\dfrac{1}{s}   \; {\bf v}^M$ with the vector $[m_1-2 \; m_2-2 \; \cdots \; m_n-2]^T$ is equal to the first component $w_1$ of the solution ${\bf w}$  of the linear system $$M {\bf w} = [m_1-2 \; m_2-2 \; \cdots \; m_n-2]^T.$$ We set  $$ M^1=M^1(m_1, m_2, \ldots, m_n)= \left[\begin{array}{cccc}
 m_1-2 & 1 & &   \\
m_2-2 & -m_2 & \ddots & \\
 \vdots & \ddots & \ddots & 1 \\
 m_n-2 & & 1 & -m_n 
\end{array}\right]$$ which is obtained by replacing the first column of $M$ by $[m_1-2 \; m_2-2 \; \cdots \; m_n-2]^T$. By Cramer's rule,   $$w_1=\frac{\det M^1}{\det M}.$$  We  claim that $\det M^1=(-1)^{n+1} (s-t-1)$ so that $$w_1 = \dfrac{\det M^1}{\det M}=\dfrac{(-1)^{n+1} (s-t-1)}{ (-1)^n s}= -1+ \dfrac{1+t}{s},$$  which finishes the proof.   To prove this last  claim, for $1 \leq i \leq n$, we set $k_i=s_i-t_i-1$ and  observe that 
\[
k_i=m_i(k_{i-1}+1)-k_{i-2}-2
\] 
for any $ 2 \leq i \leq n$, by the above recursive relations for $s_i$'s and $t_i$'s. Now we set 
\[
d'_i=\det M^1(m_1, m_2, \dots, m_i),
\] 
for $1\leq i\leq n$ and $d'_0=0$.
Note that $d'_1=m_1-2$ and $d'_n=\det M^1$, by definition.  Expanding along the last columns while computing the determinants, we derive the recursive relation  $$d'_i=-m_i (d'_{i-1}+(-1)^i)-d'_{i-2}+(-1)^i 2$$  for any $ 2 \leq i \leq n$. Finally,  by comparing the recursive relations for $k_i$'s and $d'_i$'s, we see that $d'_i= (-1)^{i+1}k_i$ and in particular, for $i=n$,  we have $$ \det M^1=d'_n= (-1)^{n+1}k_n=(-1)^{n+1}(s_n-t_n-1)=(-1)^{n+1}(s-t-1).$$

To prove the next  statement in the lemma we set $M'=M(m_n, m_{n-1}, \dots, m_1)$ and by what we showed above we know that first column of $(M')^{-1}$ is  $-\dfrac{1}{s}\; {\bf v}^{M'}$. Here we use the fact that $[m_n, m_{n-1}, \dots, m_1]=s/t^*$, where $t^*$ is the inverse of $t$ mod $s$, see \cite[Lemma~A4]{OrlikWagreich77}. But the last column of $M^{-1}$ is the first column of $(M')^{-1}$ written in reversed order and the result follows since ${\bf u}^M$ is equal to $ {\bf v}^{M'}$ written in reversed order.

Finally, we show that  the first entry of  ${\bf v}^M$ is $t$, which is the last statement in the  lemma. Note that the first entry of  ${\bf v}^M$ is the absolute value of the determinant of $$M(m_2, m_3, \dots, m_n),$$ by definition. Recall that $s/t=[m_1, m_2, \ldots, m_n]$, which implies that $$\frac{s}{t}=m_1-\frac{1}{[m_2, m_3, \ldots, m_n]}=m_1-\frac{\widetilde{t}}{\widetilde{s}}=\frac{m_1\widetilde{s}-\widetilde{t}}{\widetilde{s}} $$ where we set $\widetilde{s}/\widetilde{t} =[m_2, m_3, \ldots, m_n]$. As a consequence, by the first statement of the lemma, we observe that the absolute value of the determinant of $M(m_2, m_3, \dots, m_n)$ is $\widetilde{s}$, which in turn, is equal to $t$ by the line above. 
\end{proof}

\begin{remark} \label{rem: specialcase} In Definition~\ref{def: det}, we assumed that $m_i \geq 2$ for all $1 \leq i \leq n$. Suppose that we relax this condition so that $m_1=1$ and $m_i \geq 2$ for all $2 \leq i \leq n$, and let $M=M(1, m_2, \ldots, m_n)$, and the vectors ${\bf u}^M$ and ${\bf v}^M$ be defined as in Definition~\ref{def: det}. Then  Lemma~\ref{lem: firstrow} still (partially) holds as follows. We first observe that  $[1, m_2, \ldots, m_n]=\dfrac{s}{t}$ for some uniquely determined co-prime integers $0 < s < t $ (as opposed to $0 < t < s$). Then   $\det M = (-1)^n s$, the first and last columns of $M^{-1}$ are given by   $-(1/s)   \; {\bf v}^M$ and   $-(1/s)   \; {\bf u}^M$, respectively.  
\end{remark}

\smallskip
 
Let $X_{{\bf a^1}, {\bf a^2}}$ denote the $4$-manifold obtained by Kirby diagram in Figure~\ref{fig:plumbing} 
\begin{figure}[htb]{\small
\begin{overpic}
{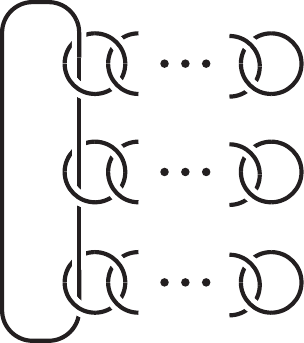}
\put(-21, 80){$-a^2_0$}
\put(42, 110){$-a^1_1$}
\put(133, 110){$-a^1_{n_1}$}
\put(42, 55){$-a^2_1$}
\put(133, 55){$-a^2_{n_2}$}
\put(42, 4){$-a^3_1$}
\put(133, 4){$-a^3_{n_3}$}
\end{overpic}}
\caption{Kirby diagram for the manifold $X_{{\bf a^1}, {\bf a^2}}$.  }
\label{fig:plumbing}
\end{figure}
so that $\partial X_{{\bf a^1}, {\bf a^2}} = M_{{\bf a^1}, {\bf a^2}}$. Then the intersection matrix $Q_{{\bf a^1}, {\bf a^2}}$ for $X_{{\bf a^1}, {\bf a^2}}$ can be described as in Figure~\ref{fig: matrix}, 
where  $$A=M(a^1_{n_1}, a^1_{n_1-1}, \ldots, a^1_1),\;\; B=M(a^2_0, a^2_1, \ldots, a^2_{n_2}), \mbox{and} \;\;  C=M(a^3_1, a^3_2, \ldots, a^3_{n_3}),$$
and in the off-diagonal terms, all entries not specified are $0$.

\begin{figure}[htb]{\small
\begin{overpic}
{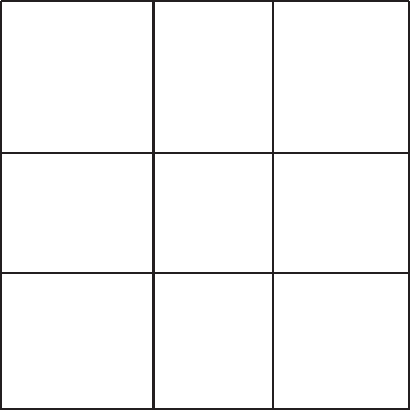}
{\Large
\put(30, 155){$A$}
\put(95, 90){$B$}
\put(158, 28){$C$}
\put(77, 128){$1$}
\put(63, 110){$1$}
\put(136, 110){$1$}
\put(76, 52){$1$}}
\end{overpic}}
\caption{Intersection matrix $Q_{{\bf a^1}, {\bf a^2}}$ for $X_{{\bf a^1}, {\bf a^2}}$.}
\label{fig: matrix}
\end{figure}

\begin{lemma} \label{lem: inverse} The inverse $Q^{-1}_{{\bf a^1}, {\bf a^2}}$ can be described as in Figure~\ref{fig: inversematrix}, where the block matrices 
$\widetilde{B}$, $D$, $E$, $F$, $G$, $H$  in the figure  are defined as follows: 
\[
G= -\dfrac{q}{(\widetilde{p})^2(p-q)}  {\bf u}^A ( {\bf u}^A)^T,  \;\;\; 
D = -\dfrac{1}{\widetilde{p} (p-q)}   {\bf u}^A ({\bf v}^B)^T,  \;\;\; 
E = -\dfrac{q}{(\widetilde{p})^2(p-q)}   {\bf u}^A  ({\bf v}^C)^T
\] 
\[ 
\widetilde{B}=M(a^2_0-1, a^2_1, \ldots, a^2_{n_2}), \;\;\; 
F=  -\dfrac{1}{\widetilde{p} (p-q)}   {\bf v}^B ({\bf v}^C)^T, \;\;\; 
H= -\dfrac{q}{(\widetilde{p})^2(p-q)}   {\bf v}^C ({\bf v}^C)^T .
\] 
\end{lemma} 

\begin{figure}[htb]{\small
\begin{overpic}
{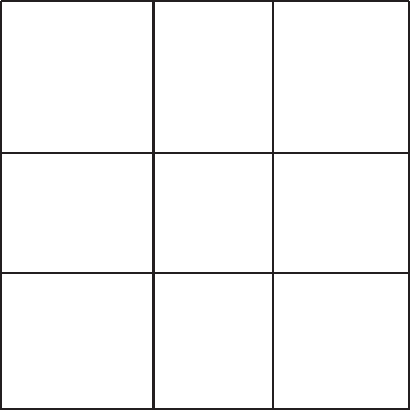}
{\Large
\put(10, 154){$A^{-1}+G$}
\put(91, 91){$\widetilde{B}^{-1}$}
\put(138, 28){$C^{-1}+H$}
\put(95, 154){$D$}
\put(27, 91){$D^T$}
\put(157, 91){$F$}
\put(95, 28){$F^T$}
\put(157, 154){$E$}
\put(27, 28){$E^T$}
}
\end{overpic}}
\caption{The matrix $Q^{-1}_{{\bf a^1}, {\bf a^2}}$.}
\label{fig: inversematrix}
\end{figure}

\begin{remark} \label{rem: row} Note that in the definition of the matrix $E$, for example, ${\bf u}^A ( {\bf v}^C)^T$ is the $n_1 \times n_3$ matrix obtained by  multiplying the $n_1 \times 1$ column matrix ${\bf u}^A$ with the $1 \times n_3$ row matrix $({\bf v}^C)^T$. Since the first component of ${\bf u}^A$ is equal to $1$, by definition, this means that the first row of the matrix $E$ is given by $ -\dfrac{q}{(\widetilde{p})^2(p-q)}  ({\bf v}^C)^T$ and the $j$th row of $E$ is obtained by multiplying its first row by the $j$th component of ${\bf u}^A$.  A similar discussion applies to the matrices $D,  F,  G, H$ and $D^T, E^T, F^T$ as well.  
\end{remark} 

The proof of Lemma~\ref{lem: inverse} is rather long but simply a straightforward computation. So we leave the proof to the reader. 

By definition (see \cite{Gompf98}), 
\[
\theta(\xi_{can}) = c_1^2 (X_{{\bf a^1}, {\bf a^2}}) - 3 \sigma (X_{{\bf a^1}, {\bf a^2}}) - 2 \chi (X_{{\bf a^1}, {\bf a^2}}),
\] 
where $\sigma (X_{{\bf a^1}, {\bf a^2}})=-(n_1+n_2+n_3+1)$ since the Kirby diagram in Figure~\ref{fig:plumbing} is negative-definite and $\chi (X_{{\bf a^1}, {\bf a^2}})=n_1+n_2+n_3+2$ because the $4$-manifold $X_{{\bf a^1}, {\bf a^2}}$ consists of a single zero handle and $n_1+n_2+n_3+1$ two handles. Therefore, to finish the proof of Proposition~\ref{prop: thetaSFS}, we need to calculate the term 
\[
c_1^2 (X_{{\bf a^1}, {\bf a^2}})={\bf r}^T_{can} Q^{-1}_{{\bf a^1}, {\bf a^2}} {\bf r}_{can},
\] 
where ${\bf r}_{can}$ denotes  the rotation vector corresponding to the canonical contact structure $\xi_{can}$. Note that ${\bf r}_{can}$  is given by ${\bf x}+{\bf y}$, where 
\[
{\bf x}=[\underbrace{0 \; \cdots \; 0}_{n_1} \;\; a^2_0-2 \;\; a^2_1-2 \; \; \cdots \;\; a^2_{n_2}-2 \;\; \underbrace{0 \; \cdots \;0}_{n_3}]^T
\]
\[
{\bf y}=[a^1_{n_1}-2 \;\; a^1_{n_1-1}-2 \;\; \cdots \; \; a^1_1-2 \; \;  \underbrace{0 \; \cdots \; 0}_{n_2+1} \;\; a^3_1-2 \; \; a^3_2-2 \;\;   \cdots \; \;a^3_{n_3}-2 ]^T.
\] 

\begin{lemma} \label{lem: dpq} We have 
\begin{equation} \label{eq: dpq} 
{\bf x}^T Q^{-1}_{{\bf a^1}, {\bf a^2}} {\bf x} = 2n_2+3-(a^2_0 +a^2_1+\cdots+a^2_{n_2}) - \dfrac{1}{[a^2_{n_2}, a^2_{n_2-1}, \ldots, a^2_1, a^2_0-1]}.\end{equation} 
\end{lemma}

\begin{proof} The calculation is the same as for the case of the spherical $3$-manifold $D(p,q)$, for which the formula above was derived  in \cite[Lemma 9.7]{EtnyreOzbagciTosun2025}. \end{proof} 
Next we set $$\alpha = -1 + \dfrac{1+\widetilde{q}}{\widetilde{p}}, \;\; \mbox{and} \;\; \beta = -1 + \dfrac{1+\widetilde{p}-\widetilde{q}}{\widetilde{p}}= \dfrac{1-\widetilde{q}}{\widetilde{p}}, $$ and observe that $\alpha +\beta= \dfrac{2-\widetilde{p}}{\widetilde{p}}$. 

\begin{lemma} \label{lem: xy} 
We have 
\begin{equation} \label{eq: xy}  
{\bf x}^T Q^{-1}_{{\bf a^1}, {\bf a^2}} {\bf y} = {\bf y}^T Q^{-1}_{{\bf a^1}, {\bf a^2}} {\bf x} = (\alpha+\beta) \left(1 - \dfrac{1}{p-q}\right) =\dfrac{2-\widetilde{p}}{\widetilde{p}}\left(1 - \dfrac{1}{p-q}\right). 
\end{equation}
\end{lemma}

\begin{proof} It is clear by symmetry,  that ${\bf x}^T Q^{-1}_{{\bf a^1}, {\bf a^2}} {\bf y} = {\bf y}^T Q^{-1}_{{\bf a^1}, {\bf a^2}} {\bf x}$.   
Now let 
\[
{\bf y_{a^1}}=[a^1_{n_1}-2 \;\; a^1_{n_1-1}-2 \;\; \cdots \;\; a^1_1-2 \;\;   \underbrace{0 \; \cdots \; 0}_{n_2+1+n_3}]^T
\] 
and 
\[
{\bf y_{a^3}}=[ \underbrace{0 \; \cdots \; 0}_{n_1+n_2+1} \;\; a^3_1-2 \; \; a^3_2-2 \; \;  \cdots \; a^3_{n_3}-2]^T
\] 
so that $ {\bf y} = {\bf y_{a^1}} + {\bf y_{a^3}}.$ We claim that 
\[
{\bf y}^T_{\bf a^1} Q^{-1}_{{\bf a^1}, {\bf a^2}} {\bf x}= \alpha \left(1 - \dfrac{1}{p-q}\right)\;\; 
\mbox{and}\;\; {\bf y}^T_{\bf a^3} Q^{-1}_{{\bf a^1}, {\bf a^2}} {\bf x}= \beta \left(1 - \dfrac{1}{p-q}\right),
\] 
which finishes the proof of Lemma~\ref{lem: xy}. To prove the former of these claims, we observe that only the matrix $D= -\dfrac{1}{\widetilde{p} (p-q)}   {\bf u}^A ({\bf v}^B)^T,$ in Figure~\ref{fig: inversematrix} is involved in the calculation. We also note that $(p-q)/q= [a^2_0-1, a^2_1, \ldots, a^2_{n_2}]$ since $p/q= [a^2_0, a^2_1, \ldots, a^2_{n_2}]$, and ${\bf v}^B = {\bf v}^{\widetilde{B}}$. Hence it follows from Lemma~\ref{lem: firstrow} that,  the first row of $(\widetilde{B})^{-1}$ is equal to 
\[ 
-\dfrac{1}{p-q} ({\bf v}^{\widetilde{B}})^T= -\dfrac{1}{p-q} ({\bf v}^B)^T .
\]
Moreover,  the dot product of  $-\dfrac{1}{p-q} ({\bf v}^{\widetilde{B}})^T$ with  $[a^2_0-3 \;\; a^2_1-2 \; \; \cdots \;\; a^2_{n_2}-2]^T$ is equal to $-1+\dfrac{1+q}{p-q}$ by Lemma~\ref{lem: firstrow}. Hence the dot product of  $-\dfrac{1}{p-q} ({\bf v}^B)^T$ with $[a^2_0-2 \;\; a^2_1-2 \; \; \cdots \;\; a^2_{n_2}-2]^T$ is equal to $-1+\dfrac{1+q}{p-q} -\dfrac{q}{p-q}= -1 + \dfrac{1}{p-q}$. This is because the only difference between these dot products is the first entry of $-\dfrac{1}{p-q} ({\bf v}^B)^T$, which is equal to $-\dfrac{q}{p-q}$ since the first entry of $({\bf v}^B)^T$ is $q$, which follows by the last statement in Lemma~\ref{lem: firstrow}. As pointed out in Remark~\ref{rem: row}, the rows of $D$ are the multiples of  $-\dfrac{1}{p-q} ({\bf v}^B)^T$ by the components of the vector  $(-1/\widetilde{p}) {\bf u}^A $. This implies that 
\[
{\bf y}^T_{\bf a^1} Q^{-1}_{{\bf a^1}, {\bf a^2}} {\bf x}= \alpha \left(1 - \dfrac{1}{p-q}\right),
\] 
by the following argument.  Recall that $A=M(a^1_{n_1}, a^1_{n_1-1}, \ldots, a^1_1)$, by definition and let $A'=  M (a^1_1, a^1_2, \ldots, a^1_{n_1})$.  Then the vector ${\bf v}^{A'}$ is obtained from ${\bf u}^{A}$ by reversing the order of its components. Thus, the dot product of $[a^1_{n_1}-2 \;\; a^1_{n_1-1}-2 \;\; \cdots \;\; a^1_1-2]^T$ (which is obtained by truncating ${\bf y}_{\bf a^1}$ by removing the zeros at the end) with the vector $(1/ \widetilde{p})  \; {\bf u}^A $ is equal to the dot product of $[a^1_1-2 \;\; a^1_{2}-2 \;\; \cdots \;\; a^1_{n_1}-2]^T$ with the vector $(1/ \widetilde{p})  \; {\bf v}^{A'}$, which in turn, is equal to  $-\left(-1 + \dfrac{1+\widetilde{q}}{\widetilde{p}}\right)=-\alpha$ by Lemma~\ref{lem: firstrow}, since $\widetilde{p}/\widetilde{q}=[a^1_1, a^1_2,  \ldots, a^1_{n_1}]$.  

 A similar argument, involving the matrix $F^T$  in Figure~\ref{fig: inversematrix}, proves that $${\bf y}^T_{\bf a^3} Q^{-1}_{{\bf a^1}, {\bf a^2}} {\bf x}= \beta \left(1 - \dfrac{1}{p-q}\right),$$ again by  Lemma~\ref{lem: firstrow}. 
\end{proof}

\begin{lemma} \label{lem: 3items} We have 
\begin{enumerate}[label=(\roman*)]
        \item   ${\bf y}^T_{\bf a^1} Q^{-1}_{{\bf a^1}, {\bf a^2}} {\bf y_{a^3}} = {\bf y}^T_{\bf a^3} Q^{-1}_{{\bf a^1}, {\bf a^2}} {\bf y_{a^1}} =-\alpha \beta\dfrac{q}{p-q},$
        \item ${\bf y}^T_{\bf a^1} Q^{-1}_{{\bf a^1}, {\bf a^2}} {\bf y_{a^1}} = -2\alpha - \alpha^2 \dfrac{q}{p-q} -(n_3-1) + \dfrac{\widetilde{q} - (\widetilde{q})^*}{\widetilde{p}},$ and
        \item ${\bf y}^T_{\bf a^3} Q^{-1}_{{\bf a^1}, {\bf a^2}} {\bf y_{a^3}} = -2\beta - \beta^2 \dfrac{q}{p-q} -(n_1-1) + \dfrac{(\widetilde{q})^* - \widetilde{q}}{\widetilde{p}},$
    \end{enumerate}
where $(\widetilde{q})^*$ is the multiplicative inverse of $\widetilde{q}$ mod $\widetilde{p}$. 
\end{lemma}

\begin{proof}  To prove the equalities stated in Item~(i), we first observe by symmetry that  
\[
{\bf y}^T_{\bf a^1} Q^{-1}_{{\bf a^1}, {\bf a^2}} {\bf y_{a^3}} = {\bf y}^T_{\bf a^3} Q^{-1}_{{\bf a^1}, {\bf a^2}} {\bf y_{a^1}},
\] 
and hence it suffices to prove that  ${\bf y}^T_{\bf a^1} Q^{-1}_{{\bf a^1}, {\bf a^2}} {\bf y_{a^3}}  =- \alpha \beta \dfrac{q}{p-q}$. It is clear that only the matrix $E$ in Figure~\ref{fig: inversematrix} is involved in the calculation. Note that the first row of $E$ is given by $-\dfrac{q}{(\widetilde{p})^2(p-q)} ({\bf v}^C)^T$ and moreover, by Lemma~\ref{lem: firstrow}, the dot product of  $({\bf v}^C)^T$ with $[a^3_1-2 \; \; a^3_2-2 \; \;  \cdots \; a^3_{n_3}-2]^T$ (which is obtained by truncating ${\bf y_{a^3}}$  by removing the zeros at the beginning) gives $-1+ \widetilde{q}$. Therefore,  the product of the first row of $E$ with $[a^3_1-2 \; \; a^3_2-2 \; \;  \cdots \; a^3_{n_3}-2]^T$  is equal to  $\dfrac{1- \widetilde{q}}{(\widetilde{p})^2} \dfrac{q}{p-q}.$ Since all the rows of $E$ are given by multiples of the first row by the components of  the vector $ {\bf u}^A$, and the dot product of $[a^1_{n_1}-2 \;\; a^1_{n_1-1}-2 \;\; \cdots \;\; a^1_1-2]^T$ (which is obtained by truncating ${\bf y}_{\bf a^1}$ by removing the zeros at the end) with the vector ${\bf u}^A$ is equal to  $-(1+ \widetilde{q} -\widetilde{p})$ as observed in the proof of Lemma~\ref{lem: xy},  it follows that 
\[
{\bf y}^T_{\bf a^1} Q^{-1}_{{\bf a^1}, {\bf a^2}} {\bf y_{a^3}} = -(1+ \widetilde{q} -\widetilde{p}) \dfrac{1- \widetilde{q}}{(\widetilde{p})^2} \dfrac{q}{p-q}= - \alpha \beta \dfrac{q}{p-q},
\] 
which finishes the proof of Item~(i).

To prove Item~(ii), it is clear that only the top-left block $A^{-1} + G$  in Figure~\ref{fig: inversematrix} is involved in the calculation. We have 
\begin{equation} \label{eq: yay}
\begin{split}
{\bf y}^T_{\bf a^1} Q^{-1}_{{\bf a^1}, {\bf a^2}} {\bf y_{a^1}} &=  [a^1_{n_1}-2 \;\;  \cdots \;\; a^1_1-2] (A^{-1} + G) [a^1_{n_1}-2 \;\; \cdots \;\; a^1_1-2]^T  \\
 & =[a^1_{n_1}-2 \;\; \cdots \;\; a^1_1-2] A^{-1} [a^1_{n_1}-2 \;\; \cdots \;\; a^1_1-2]^T\\  & 
 \quad\quad +[a^1_{n_1}-2 \;\; \cdots \;\; a^1_1-2] G [a^1_{n_1}-2 \;\;  \cdots \;\; a^1_1-2]^T. \\ \end{split}
\end{equation}
We claim that the first term on the right in Equation ~\eqref{eq: yay} is equal to 
\[ 
-2\alpha  -(n_3-1) + \dfrac{\widetilde{q} - (\widetilde{q})^*}{\widetilde{p}},
\] 
and the second term is equal to 
\[
- \alpha^2 \dfrac{q}{p-q},
\] 
which together proves the formula in Item~(ii). To prove the first claim, we use the well-known fact that $\widetilde{p}/\widetilde{q}=[a^1_1, a^1_2, \ldots, a^1_{n_1}]$  if and only if $[a^1_{n_1}, a^1_{n_1-1}, \ldots, a^1_1]=  \widetilde{p} /  (\widetilde{q})^*$, where $(\widetilde{q})^*$ is  the multiplicative inverse of  $\widetilde{q}$ mod $\widetilde{p}$, see \cite[Lemma~A4]{OrlikWagreich77}. Next we observe that the first term 
\[
[a^1_{n_1}-2 \;\; \cdots \;\; a^1_1-2] A^{-1} [a^1_{n_1}-2 \;\; \cdots \;\; a^1_1-2]^T
\] 
on the right of Equation~(\ref{eq: yay}) can be calculated as follows. By Proposition 9.3 in our earlier work \cite{EtnyreOzbagciTosun2025}, we know that $$\theta(\xi_{can})=   -\dfrac{2+\widetilde q+ (\widetilde{q})^*}{\widetilde{p}} -I(\widetilde{p} /  \widetilde{q}) $$ for the canonical contact structure $\xi_{can}$ on the lens space $L( \widetilde{p},  (\widetilde{q})^*)$. By definition of $\theta(\xi_{can})$, we have  $$c_1^2 = \theta(\xi_{can})+3 \sigma +2 \chi= \theta(\xi_{can})+3 (-n_1) 
+2(n_1+1)=\theta(\xi_{can})-n_1+2.$$ We observe that $c_1^2 = [a^1_{n_1}-2 \;\; \cdots \;\; a^1_1-2] A^{-1} [a^1_{n_1}-2 \;\; \cdots \;\; a^1_1-2]^T$ and combining the last two equations above, and using the definition of $I(\widetilde{p} /  \widetilde{q})$, we see that the first term on the right of Equation~(\ref{eq: yay}) is equal to
\begin{equation} \label{eq: contfrac}
 -\dfrac{2+\widetilde q+ (\widetilde{q})^*}{\widetilde{p}} - \sum_{i=1}^{n_1} (a^1_i -3) -n_1+2. 
 \end{equation}
 Next, we observe the general fact that 
\[
I(\widetilde{p} /  \widetilde{q})= \sum_{i=1}^{n_1} (a^1_i -3) = n_3-n_1-1,
\] 
where $n_3$ is the length of the continued fraction $\widetilde{p}/(\widetilde{p}-\widetilde{q})=[a^3_1, a^3_2, \ldots, a^3_{n_3}]$ dual to $\widetilde{p}/\widetilde{q}=[a^1_1, a^1_2, \ldots, a^1_{n_1}]$, which can be proved by induction on $n_1$ using the Riemenschneider point rule, \cite{Riemenschneider1974}. Plugging this back in the Equation~\eqref{eq: contfrac}, we see that the first term in Equation~\eqref{eq: yay} is equal to 
\begin{equation} \label{eq: first}
 -\dfrac{2+\widetilde q+ (\widetilde{q})^*}{\widetilde{p}} + 3-n_3.
 \end{equation}
 On the other hand, 
 \begin{equation} \label{eq: alpha}
\begin{split}
 -2\alpha  -(n_3-1) + \dfrac{\widetilde{q} - (\widetilde{q})^*}{\widetilde{p}} &=  -2\left(-1 + \dfrac{1+\widetilde{q}}{\widetilde{p}} \right)  -(n_3-1) + \dfrac{\widetilde{q} - (\widetilde{q})^*}{\widetilde{p}}  \\
 & = -\dfrac{2+q+ (\widetilde{q})^*}{\widetilde{p}} + 3-n_3.  \end{split}
\end{equation}
By comparing Equation~\eqref{eq: first} and  Equation~(\ref{eq: alpha}), we conclude that the first term in the Equation~\eqref{eq: yay} is as claimed. 

To prove the second claim, we observe that the first row of $G$ is $-\dfrac{q}{(\widetilde{p})^2(p-q)} {\bf u}^A$ and the dot product of ${\bf u}^A$ with  $[a^1_{n_1}-2 \;\; a^1_{n_1-1}-2 \;\; \cdots \;\; a^1_1-2]^T$ is equal to $1 + \widetilde{q} - \widetilde{p}$, as observed in the proof of Lemma~\ref{lem: xy}.  It follows that the second term 
 \[
[a^1_{n_1}-2 \;\; \cdots \;\; a^1_1-2] G [a^1_{n_1}-2 \;\;  \cdots \;\; a^1_1-2]^T
 \]  
 in the Equation~\eqref{eq: yay} is equal to 
 \[
 -(1 + \widetilde{q} - \widetilde{p})^2 \dfrac{q}{(\widetilde{p})^2(p-q)}=- \alpha^2 \dfrac{q}{p-q}.
 \]
 
 The proof of Item~(iii) is very similar to Item~(ii), where only the bottom-right block $C^{-1} + H$  in Figure~\ref{fig: inversematrix} is involved in the calculation. 
\end{proof}

\begin{lemma} \label{lem: yy} We have 
$${\bf y}^T Q^{-1}_{{\bf a^1}, {\bf a^2}} {\bf y}  = 2\left(\dfrac{\widetilde{p} -2}{\widetilde{p}} \right)- (n_1+n_3-2) - \left(\dfrac{\widetilde{p} -2}{\widetilde{p}} \right)^2 \dfrac{q}{p-q}$$
\end{lemma}

\begin{proof} It follows immediately from Lemma~\ref{lem: 3items} that 
\begin{equation} \label{eq: y}
\begin{split}
{\bf y}^T Q^{-1}_{{\bf a^1}, {\bf a^2}} {\bf y}  & =({\bf y_{a^1}} + {\bf y_{a^3}})^T Q^{-1}_{{\bf a^1}, {\bf a^2}}  ({\bf y_{a^1}} + {\bf y_{a^3}}) \\
 & ={\bf y}^T_{\bf a^1} Q^{-1}_{{\bf a^1}, {\bf a^2}} {\bf y_{a^1}} + {\bf y}^T_{\bf a^1} Q^{-1}_{{\bf a^1}, {\bf a^2}} {\bf y_{a^3}} + {\bf y}^T_{\bf a^3} Q^{-1}_{{\bf a^1}, {\bf a^2}} {\bf y_{a^1}} + {\bf y}^T_{\bf a^3}Q^{-1}_{{\bf a^1}, {\bf a^2}} {\bf y_{a^3}}. \\  
 & = -2 (\alpha + \beta) - (n_1+n_3-2) - (\alpha + \beta)^2 \dfrac{q}{p-q} \\
 & = 2\left(\dfrac{\widetilde{p} -2}{\widetilde{p}} \right)- (n_1+n_3-2) - \left(\dfrac{\widetilde{p} -2}{\widetilde{p}} \right)^2 \dfrac{q}{p-q}.
\end{split}
\end{equation}
\end{proof}

\begin{proof}[Proof of Proposition~\ref{prop: thetaSFS}] By combining Lemma~\ref{lem: dpq},  Lemma~\ref{lem: xy}, and  Lemma~\ref{lem: yy}, we obtain
\begin{equation} \label{eq: c1}
\begin{split}
c_1^2 (X_{{\bf a^1}, {\bf a^2}}) & ={\bf r}^T_{can} Q^{-1}_{{\bf a^1}, {\bf a^2}} {\bf r}_{can}\\
 & =({\bf x}+{\bf y})^T Q^{-1}_{{\bf a^1}, {\bf a^2}} ({\bf x}+{\bf y}) \\
 & =  {\bf x}^T Q^{-1}_{{\bf a^1}, {\bf a^2}} {\bf x} + {\bf x}^T Q^{-1}_{{\bf a^1}, {\bf a^2}} {\bf y} + {\bf y}^T Q^{-1}_{{\bf a^1}, {\bf a^2}} {\bf x} + 
{\bf y}^T Q^{-1}_{{\bf a^1}, {\bf a^2}} {\bf y} \\ & = {\bf x}^T Q^{-1}_{{\bf a^1}, {\bf a^2}} {\bf x} +2 {\bf x}^T Q^{-1}_{{\bf a^1}, {\bf a^2}} {\bf y} + 
{\bf y}^T Q^{-1}_{{\bf a^1}, {\bf a^2}} {\bf y}  \\
 & =  2n_2+3-(a^2_0 +a^2_1+\cdots+a^2_{n_2}) - \dfrac{1}{[a^2_{n_2}, a^2_{n_2-1}, \ldots, a^2_1, a^2_0-1]} \\ 
 & \quad \quad +  \dfrac{2(2-\widetilde{p})}{\widetilde{p}} \left(1-\dfrac{1}{p-q}\right) - \left(\dfrac{\widetilde{p} -2}{\widetilde{p}} \right)^2 \dfrac{q}{p-q} -  (n_1+n_3-2) + \dfrac{2(\widetilde{p} -2)}{\widetilde{p}}.  \end{split}
\end{equation}

Therefore, 
\begin{equation} \label{eq: theta}
\begin{split}
\theta (\xi_{can}) & = c_1^2 (X_{{\bf a^1}, {\bf a^2}}) - 3 \sigma (X_{{\bf a^1}, {\bf a^2}})- 2 \chi (X_{{\bf a^1}, {\bf a^2}}) \\
& =  c_1^2 (X_{{\bf a^1}, {\bf a^2}}) + n_2+n_1+n_3-1 \\
 & =  1- I(p/q)-\dfrac{1}{[a^2_{n_2}, a^2_{n_2-1}, \ldots, a^2_1, a^2_0-1]} +\dfrac{2(\widetilde{p}-2)}{\widetilde{p}(p-q)} -\dfrac{(\widetilde{p}-2)^2 q}{(\widetilde{p})^2(p-q)},
 \end{split}
\end{equation}
as claimed. 
\end{proof}

\section{Non-balanced contact structures and spherical $3$-manifolds} 

Recall that Proposition~\ref{dihedO}, in light of Theorem~\ref{thm: nonneg}, says that if $\xi$ is not a balanced contact structure on the small Seifert fibered space $Y(e_0; \frac{1}{2},  \frac{q}{p}, \frac{1}{2})$ with $e_0\geq 0$, then it is not filled by a symplectic rational homology ball. We depicted the plumbing diagram of  $Y(e_0; \frac{1}{2},  \frac{q}{p}, \frac{1}{2})$ in  Figure~\ref{dO}, for the convenience  of the reader, where  $p/q=[a_1, a_2, \ldots, a_k]$. 

\begin{figure}[htb]{
\begin{overpic}
{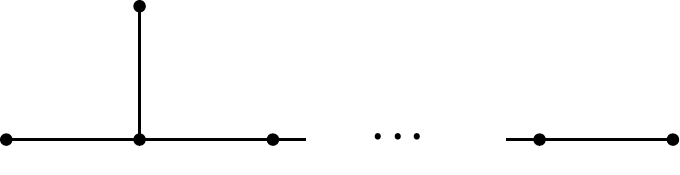}
\put(63, 2){$e_0$}
\put(120, 2){$-a_1$}
\put(247, 2){$-a_{k-1}$}
\put(313, 2){$-a_k$}
\put(-8, 2){$-2$}
\put(75, 79){$-2$}
\end{overpic}}
 \caption{The plumbing diagram of  the small Seifert fibered space $Y(e_0; \frac{1}{2},  \frac{q}{p}, \frac{1}{2})$, where $p/q=[a_1, a_2, \ldots, a_k]$.}
  \label{dO}
\end{figure}

We observe that there are exactly two tight contact structures, which we denote by $\xi^{\pm}$, on $Y(e_0; \frac{1}{2},  \frac{q}{p}, \frac{1}{2})$  that are not balanced. These are the contact structures where in their contact surgery presentation, the complementary legs are stabilized once consistently (say both positive) and the non-complementary leg is also stabilized consistently, either all positive or all negative.  The contact structure $\xi^-$(resp $\xi^+$) is when the non-complementary leg has opposite (respectively the same) signs as the complementary legs.  Our goal is then to prove that $(Y(e_0; \frac{1}{2},  \frac{q}{p}, \frac{1}{2}) , \xi^{\pm})$ does not admit a rational homology ball symplectic filling. (If one makes the opposite choice for the sign of the complementary legs, then we would get $-\xi^\pm$, but these have the same $\theta$-invariant so we will not consider them.)
This will be a corollary of the following formulas for $\theta(\xi^\pm)$. 

\begin{lemma}\label{dihedralC} Let $\frac{p'}{q'}=[a_1, \ldots, a_{k-1}]$, and assume that $p'=1, q'=0$ when $k=1$. Then
\begin{enumerate}
 \item   $\theta(\xi^-)=-(a_1+\cdots +a_k-(3k+e_0-2))-\frac{(e_0+1)p'+q'}{(e_0+1)p+q}$, and 
\item   $\theta(\xi^+)=-(a_1+\cdots +a_k-(3k+e_0-1))-\frac{(e_0+1)p'+q'}{(e_0+1)p+q}-\frac{(e_0-3)p+q+4}{(e_0+1)p+q}$.  
\end{enumerate}
\end{lemma}

\begin{proof} The proof of Lemma~\ref{dihedralC} is an easier version of the calculations that were done in Section~\ref{subsec: theta}, and therefore, we just provide the setup and leave some straightforward details to the interested reader. 

We start by describing $\xi^{\pm}$ by appropriate contact surgery diagrams. Note that the plumbing diagram of $Y(e_0; \frac{1}{2},  \frac{q}{p}, \frac{1}{2})$ depicted in Figure~\ref{dO}  is not immediately suitable for describing  contact surgery diagrams of $\xi^{\pm}$. To remedy this we apply a sequence of blow-ups between the central vertex and the one standing next to it on its right-hand side, until the central curve has framing zero. This adds an additional $e_0$ many vertices (2 handles) to the leg corresponding to the singular fiber. After this, we can realize the manifold as a Stein handlebody diagram by converting the 0-framed curve to a Stein $1$-handle and attaching to it $k+e_0+2$ Stein $2$-handles with framings 
\[
-2, -2, \underbrace{-1, -2, \ldots, -2}_{e_0}, -a_1-1, -a_2, \ldots, -a_k,
\] 
respectively.  To obtain contact surgery diagrams describing $\xi^{\pm}$, we convert the Stein $1$-handle into a contact $(+1)$-surgery (this still describes Stein fillable contact structures). Let $(X, J^{\pm})$ denote this latter manifold where the complementary legs are stabilized the same way (say positive) and the non-complementary leg includes stabilizations that are either all positive or all negative. In this description, the Stein structures $J^{\pm}$ induce $\xi^{\pm}$ on the boundary.

We observe that $\chi(X)=k+e_0+4$, since the handle-decomposition of $X$ consists of  $k+e_0+3$ two-handles and a zero-handle. As for the signature we first note that $Y=Y(e_0; \frac{1}{2},  \frac{q}{p}, \frac{1}{2})$  is an L-space, and the plumbing graph corresponding to $-Y$ is negative definite, since it arises as the {\em oriented} link of quotient surface singularity. (Indeed, $-Y$ will be the dihedral-type spherical $3$-manifold $D(\widetilde{p}, \widetilde{q})$, canonically oriented as the link of a complex surface singularity, where $\widetilde{p}=(e_0+2)p+q$ and $\widetilde{q}=p$.)
So, as argued in \cite[Section~8.1]{StipsiczSzaboWahl08}, it must be that $b^+(X)=1$, and hence $\sigma(X)=-(k+e_0+1)$. The main ingredient of the proof of Lemma~\ref{dihedralC} is the  calculation of the square of the first Chern class 
\begin{equation}\label{chern}
c^2_1(X, J^{\pm})={\bf r}_{\pm}^{~T} I_X^{-1} {\bf r}_{\pm}, 
\end{equation}
where ${\bf r}_{\pm}=\begin{bmatrix}
0 & 1&1& 0 \cdots 0&\pm (a_1-1)& \pm (a_2-2) \cdots  & \pm(a_k-2)
\end{bmatrix}^T$ is the rotation vector, and $I_X$ is the intersection matrix of $X$ described above. Instead of computing  $I_X^{-1}$ similar to our calculations in  Section~\ref{subsec: theta}, it is much easier here to solve the linear system $I_X{\bf x}=\rot_{\pm}$ for ${\bf x}$. This will result in the computations
\[
c_1^2(X, J^-)=-(a_1+\cdots +a_k-(2k-1)) -\dfrac{1}{[a_k, \ldots, a_2, a_1+1, \underbrace{2,\ldots, 2}_{e_0}]}
\]
\noindent where we note that $$\frac{(e_0+1)p'+q'}{(e_0+1)p+q}=\dfrac{1}{[a_k, \ldots, a_2, a_1+1, 2,\ldots, 2]},$$ and
\[
c_1^2(X,J^+)=-(a_1+\cdots+a_k-2k)-\frac{(e_0+1)p'+q'}{(e_0+1)p+q}-\frac{(e_0-3)p+q+4}{(e_0+1)p+q},
\]
which in turn, yields the formulas in the lemma. The details of this calculation are left to the reader. 
\end{proof} 

\begin{proof}[Proof of Proposition~\ref{dihedO}] 
If a contact structure on a Seifert fibered space with $e_0\geq 0$ is balanced, then whether or not it symplectically bounds a rational homology ball is determined by Theorem~\ref{thm: nonneg}. So we are left to show that if the contact structure is not balanced, then it does not symplectically bound a rational homology ball. 

As we mentioned above, our proof will rely on  Lemma~\ref{dihedralC}. First, for any $e_0\geq 0$ and $\frac{p}{q}>1$, one can easily see that, $c_1^2(X, J^-) \notin \Z$ (and hence $\theta(\xi^-)\notin \Z$), since 
\[
\frac{(e_0+1)p'+q'}{(e_0+1)p+q}=\frac{1}{[a_k,\ldots, a_2,  a_1+1, 2, \ldots, 2]}<1.
\]

\noindent We note a similar equality was established in the proof of Proposition~\ref{prop: minSFS}.
In particular,  $\xi^{-}$ cannot be symplectically filled by a rational homology ball. 

Although  $\theta(\xi^{+})$ may take integer values, with the help of Theorem~\ref{lecuonas}, we will show that $\theta(\xi^+)\neq -2$. Recall that Theorem~\ref{lecuonas} characterizes exactly which small Seifert fibered spaces with complementary legs admit smooth rational homology ball fillings. To apply the theorem we first need to perform a sequence of $(-e_0-1)$ Rolfsen twists along the singular fiber with framing $-\frac{p}{q}$. The new framing for the singular fiber will be  \[-\frac{p}{(e_0+1)p+q}=-1+ \frac{1}{[\underbrace{2, \ldots, 2}_{e_0}, a_1+1, a_2, \ldots, a_k]}.\]

\noindent Now Theorem~\ref{lecuonas} says this Seifert fibered space bounds a smooth rational homology ball exactly when $r/s=[\underbrace{2, \ldots, 2}_{e_0}, a_1+1, a_2, \ldots, a_k]\in \mathcal R$. According to Lisca \cite{Lisca2007} the integer $I(r/s)$ satisfies \[I(r/s)=\sum_{i=1}^{e_0} (2-3)+(a_1+1-3)+\sum_{j=2}^k(a_j-3)\leq 1.\]  Indeed Lisca also proves that $I(r/s)=1$ exactly when $\frac{r}{s}=\frac{m^2}{mh-1}$ for some relatively prime integers $0<h<m$. 

We can expand and calculate this integer as $I(r/s)=-3k-e_0+1+(a_1+\cdots +a_k)$, so that Item~(2) of Lemma~\ref{dihedralC} reads as 
\begin{equation}\label{thetaC}
\theta(\xi^+)=-I(r/s)-\frac{(e_0+1)p'+q'}{(e_0+1)p+q}-\frac{(e_0-3)p+q+4}{(e_0+1)p+q}.   
\end{equation}

Note that the fraction terms in the formula for $\theta(\xi^{+})$ are both strictly less than one. In particular, if $I(r/s)<1$, then $\theta(\xi^+)>-2$. If $I(r/s)=1$, then we know that $\frac{r}{s}=\frac{m^2}{mh-1}$ for some relatively prime integers $0<h<m$. Since the continued fraction for $r/s$ must start with a $2$, we see that $\frac{m}{2}<h$. We now compute $p/q$ explicitly by using the equality $[2, \ldots, 2, a_1+1, a_2, \ldots, a_k]=\frac{m^2}{mh-1}$. We can rewrite this as  \[[\underbrace{2, \ldots, 2}_{e_0-1}, a_1+1, a_2, \ldots, a_k]=(2-\frac{m^2}{mh-1})^{-1}=\frac{mh-1}{2(mh-1)-m^2}\] 
and repeating this $e_0-1$ more times gives that \[[a_1+1, a_2, \ldots, a_k]=\frac{e_0(mh-1)-(e_0-1)m^2}{(e_0+1)(mh-1)-e_0m^2}.\] 
 We point out that if $e_0=0$ (that is when there are no $2$'s at the beginning) the proof still works. Finally,  we ``move" the $+1$ in the first term in the continued fraction to the right-hand side to obtain
\[[a_1, a_2, \ldots, a_k]=\frac{m^2-(mh-1)}{(e_0+1)(mh-1)-e_0m^2}\] which implies that $p=m^2-(mh-1)$ and $q=(e_0+1)(mh-1)-e_0m^2=mh-1-pe_0$.

We next explicitly calculate $p'/q'$. Since $[a_1, a_2, \ldots, a_k]=\frac{m^2-(mh-1)}{(e_0+1)(mh-1)-e_0m^2}$, we have $[a_k, a_{k-1}, \ldots, a_1]=\frac{m^2-(mh-1)}{t^*}$ where $0<t^*<p$ is the inverse of $(e_0+1)(mh-1)-e_0m^2=(mh-1-pe_0)$ mod $p=m^2-mh+1$, and by \cite[Lemma~A4]{OrlikWagreich77} we know that $t^*=p'$. Indeed, one can easily check that $p'=(m-h)^2$. To see this, we first observe that $p'=(m-h)^2=p+h^2-mh-1$, and calculate 
\[
\begin{split}
p'((mh-1-pe_0))&=(p+h^2-mh-1)(mh-1-pe_0)\\
& \equiv (h^2-mh-1)(mh-1) \mod p\\
&=-h^2(m^2-mh+1)+1 \\
& \equiv 1 \mod p.
\end{split}
\] since $p=m^2-mh+1.$

Similarly one can calculate that $q'=(2e_0+1)mh-(e_0+1)h^2-e_0m^2-1$.  We emphasize a surprising observation 
\begin{equation}\label{eq: pq'}
p-2=(e_0+1)p'+q',
\end{equation}
that will be crucial below.

Recall that we want to prove that $\theta(\xi^+)\neq -2$. Since we are assuming $I(r/s)=1$, Equation ~(\ref{thetaC}) reduces to \[\theta(\xi^+)=-1-\frac{(e_0+1)p'+q'}{(e_0+1)p+q}-\frac{(e_0-3)p+q+4}{(e_0+1)p+q}. \] 
So, assuming  $\theta(\xi^+)=-2$,  we obtain the equation 
$$(e_0+1)p'+q'+4-4p=0,$$
which is impossible, since by using Equation~\ref{eq: pq'},  we can explicitly calculate that  \[(e_0+1)p'+q'+4-4p=-3p+2<0.\]

Thus, $\theta(\xi^+)\neq -2$, and we conclude that  $\xi^+$ cannot be symplectically filled by a rational homology ball, which finishes our proof.   
\end{proof}

Finally, as a byproduct of our work on small Seifert fibered spaces, we are ready to give a proof of Theorem~\ref{classifyfillings}, which explicitly describes a complete classification of contact structures on oriented spherical $3$-manifolds which are symplectically fillable by rational homology balls.

\begin{proof}[Proof of Theorem~\ref{classifyfillings}]
As noted in the introduction, we only need to determine which of the ``oppositely oriented" dihedral-type spherical $3$-manifolds admit contact structures having symplectic rational homology ball fillings. We recall that any dihedral-type spherical $3$-manifold with its canonical orientation  is of the form $Y(e_0;1/2, s,1/2)$ with $e_0\leq -2$, as a small Seifert fibered space. So any ``oppositely oriented" dihedral-type spherical $3$-manifold has the same form but with $e_0\geq -1$.  To prove our theorem, we will go through a case by case analysis below based on the possible values of $e_0$.

Recall that in Theorem~\ref{thm: minusone} and the subsequent remarks, we discussed the classification of contact structures that admit symplectic rational homology ball fillings, on a small Seifert fibered space with (uniquely) complementary legs and $e_0=-1$. For relatively prime integers $0<h<m$ and $n\geq 2$ or $m=1, h=0$ and $n\geq 1$, we set 
\begin{equation}\label{eq:Ymhn}
Y_{m,h,n}:=Y\left(-1;\dfrac{1}{2}, \dfrac{m^2}{nm^2-mh+1}, \dfrac{1}{2}\right)=-D((n+1)m^2-mh+1, nm^2-mh+1),
\end{equation} 
where the second equality is obtained by converting the small Seifert fibered space notation used in Theorem~\ref{thm: minusone}  to the dihedral-type spherical $3$-manifold notation.  Note that in Equation~(\ref{eq:Ymhn}), we intentionally excluded the case where $n=1$ and $0 < h < m$ are coprime. This is because the conversion given by the second equality does not work in that case. We will return to this case at the end of the proof and appropriately define $Y_{m,h,1}$, when $0 < h < m$ are coprime.

If $m= 1$ and  $h=0$, then the $3$-manifold $Y_{1,0,n}=-D(n+2, n+1)$ is the same as $M_{n+1}$ in Remark~\ref{extra} (1). In particular, according to that remark, $-D(n+2, n+1)$ carries exactly three, respectively two, distinct contact structures with symplectic rational homology ball fillings if $n=1$, respectively $n>1$. This proves  (c) and (d) in Item~(2) of Theorem~\ref{classifyfillings}. 

When $0<h<m$, we have three cases to consider: $n>2$, $n=2$ and $n=1$. For the first case, the $3$-manifold $Y_{m,h,n}$ is uniquely complementary as a small Seifert fibered space, by definition. 
In \cite[Theorem~$1.1$]{GhigginiLiscaStipsicz07}, a complete classification of tight contact structures on $Y_{m,h,n}$  is given along with explicit contact surgery diagrams. More specifically, that paper considered the Seifert fibered spaces $Y(-1;r_1, r_2, r_3)$ with $r_1\geq r_2\geq r_3$ and $r_1, r_2\geq 1/2$. They showed that any tight contact structure was obtained by contact surgery on the right-hand diagram in Figure~\ref{fig:LSdiagram}, and determined when such surgeries were tight and when they were distinct. When $n>2$ we see that $r_1=1/2$ and $r_2=1/2$ while $r_3=p/(np-q)\in(0,1/2)$. Thus, in the right-hand diagram in Figure~\ref{fig:LSdiagram}, we see that the third and fourth unknots must be stabilized once before performing Legendrian surgery. If they are both stabilized in the same way, then by Theorem~\ref{thmMatkovic} we see that the contact structures cannot be symplectically fillable, so we can ignore these contact structures. We are left to consider the case where these Legendrian knots are stabilized in opposite directions. 
If the continued fraction expansion of $-1/r_3=[-n, -a_1, \ldots, -a_k]$, then we know the number of possible contact surgeries on the fifth unknot on the right of Figure~\ref{fig:LSdiagram} is $n\prod_{i=1}^{k}(a_i^2-1)$. 
However, Corollary~5.4 in \cite{GhigginiLiscaStipsicz07} says that there are only $2 \prod_{i=1}^{k}(a_i^2-1)$ contact structures coming from this diagram. Finally, as we noted in the proof of Theorem~\ref{maine00}, we know that there are only $4$ of these that can possibly bound a rational homology ball (the others do not have the right $\theta$-invariant; or alternatively, they will have a mixed torus, which will prevent a rational homology ball filling). Thus, there are at most $4$ contact structures that could bound a rational homology ball, and the construction in the proof of Theorem~\ref{thm: minusone} gives $4$ such contact structures. This completes Item~(2)(b) of Theorem~\ref{classifyfillings} when $n>2$.

Now consider the case when $n=2$. In this case the surgery coefficients in the diagram on the right of Figure~\ref{fig:LSdiagram} are $r_1=p/(2p-q)\in(1/2,1), r_2=1/2,$ and $r_3=1/2$. Now the fourth and fifth unknots must be stabilized once before performing Legendrian surgery. Denote the sign of the stabilization by $s_i$. The continued fraction for $-1/r_1$ will now be of the form $[-2, -a_1,\ldots, -a_k]$. To perform contact $(-1/r_1)$ surgery, we will need to attach a chain of unknots to the third unknot in Figure~\ref{fig:LSdiagram}, stabilize the original unknot $1$ time, denote this sign by $s_1$, and Legendrian realize each of the others with $\tb=-a_i+1$ (so it will be an $(a_i-2) $-stabilized unknot). We now recall part of the classification scheme in \cite{GhigginiLiscaStipsicz07}. They break the contact structures into two types. The first, denoted $A(\xi)$, is the set of tight contact structures coming from contact surgery on the diagram when $s_1s_2=-1$, and the second, denoted $A(\Xi)$, is the set of tight contact structures when $s_1s_2=1$. The paper shows \cite[Proposition~5.2]{GhigginiLiscaStipsicz07} that $A(\xi)\cap A(\Xi)=\emptyset$. We also note that when $s_2s_3=1$, an argument almost identical to, but easier than, the one in the proof of Proposition~\ref{dihedO} will show that the $\theta$ invariants of these contact structures are not integral, and hence these contact structures do not bound a rational homology ball. So we are left to consider the case when $s_2s_3=-1$.

We begin by considering $A(\xi)$, so we know that $s_1s_2=-1$ and $s_2s_3=-1$. Then by Lemma~\ref{balanced} we see that the surgery diagram on the right of Figure~\ref{fig:LSdiagram} with the $(-1/r_1)$ framed unknot removed gives $S^1\times S^2$ with its standard tight contact structure, and then, as in Proposition~\ref{SteinCobord}, we see that there is a Stein homology cobordism from the lens space $L(p,q)$ to $Y(-1;r_1, r_2, r_3)$. There are two such cobordisms depending on the stabilization used to perform the contact $(-2)$-surgery on the third unknot in the diagram. As there are two contact structures on the lens space that bound rational homology balls, we can glue them to create $4$ symplectic rational homology ball fillings of $Y(-1;r_1, r_2, r_3)$. These can be distinguished by their $\text{spin}^c$ structures as in the proof of Proposition~5.2 and Corollary~5.4 in \cite{GhigginiLiscaStipsicz07}. Since the $\theta$-invariant of all other contact structures on the lens spaces is not $-2$, none of the other contact structures on the Seifert fibered space (in this case) can bound a rational homology ball (recall that the rational homology cobordism shows that the $\theta$-invariant of the contact structure on the Seifert fibered space is the same as on the lens space). Thus, in $A(\xi)$ there are $4$ contact structures that bound rational homology balls. 

We now turn to the contact structures in $A(\Xi)$. Here we know that $s_1s_2=1$ and $s_2s_3=-1$. As argued above, there are at most $4$ contact structures that bound rational homology balls. We will see that there are only $2$ (so the $4$ distinct surgery diagrams giving the rational homology balls only give $2$ distinct contact structures). To see this, we note that $2$ of these contact structures have $s_1=s_2=1$ and $2$ have $s_1=s_2=-1$. But Lemma~5.6 implies that the first $2$ agree with the second $2$. Thus there are only two contact structures bounding rational homology balls in $A(\Xi)$, and we have a total of exactly $6$ contact structures bounding rational homology balls, yielding the statement in Item~(2)(a) of Theorem~\ref{classifyfillings}. 

Finally, for $n=1$, and relatively prime integers $0 < h < m$  we set 
$$
Y_{m,h,1}:=Y\left(e_0; \dfrac{1}{2}, \dfrac{m^2+(e_0+1)(-m^2+mh-1)}{m^2-mh+1}, \dfrac{1}{2}\right)=-D(2m^2-mh+1, m^2-mh+1),
$$ 
where the integer $e_0 \geq 0$ is uniquely defined by the pair of integers $0 < h < m$ as in Lemma~\ref{lem:unique-e0} and the second equality is shown in Lemma~\ref{lem:conversion}. Recall that Equation~(\ref{eq:Ymhn}) excludes the case where $n=1$ and $0 < h < m$ are coprime. 

With this notation in place, Theorem~\ref{thm: nonneg}, together with Proposition~\ref{dihedO}, implies that there are exactly four contact structures on  $Y_{m, h, 1}$ that have symplectic rational homology ball fillings. This completes Item~(2)(b) of Theorem~\ref{classifyfillings} when $n=1$.
\end{proof}

\begin{lemma}\label{lem:unique-e0}
For any integers $0 < h < m$, 
there exists a unique integer $e_{0}\ge 0$ such that
\begin{equation}\label{eq:ineq-original}
\frac{-m^{2}+mh-1}{\,m^{2}+(e_{0}+1)(-m^{2}+mh-1)\,}< -1 .
\end{equation}
In fact, $e_{0}$ is uniquely determined as
\begin{equation}\label{eq:e0-formula}
e_{0}=\left\lfloor \frac{m^{2}}{m^{2}-mh+1}\right\rfloor-1.
\end{equation}
\end{lemma}

\begin{remark} The integers  $h < m$ are not necessarily coprime in Lemma~\ref{lem:unique-e0}, although that is an assumption for our applications. 
\end{remark}

\begin{proof}
Let $B:=m^{2}-mh+1$, which is positive since $0 < h < m$.  
Rewrite Inequality~\eqref{eq:ineq-original} as
\begin{equation}\label{eq:new}
\frac{-B}{\,m^{2}-(e_{0}+1)B\,}< -1.
\end{equation}
Since the numerator is negative, the left-hand side can be less than $-1$ only if the denominator is positive; hence, we necessarily have
\begin{equation}\label{eq:den-pos}
(e_{0}+1)B<m^{2}.
\end{equation}
Multiplying the Inequality~\eqref{eq:new} by its positive denominator gives
\[
-B<-\bigl(m^{2}-(e_{0}+1)B\bigr),
\]
which is equivalent to 
\[
(e_{0}+2)B>m^{2}.
\]
Together with Inequality~\eqref{eq:den-pos} this shows that Inequality~\eqref{eq:ineq-original} is
equivalent to the strict double inequality
\begin{equation}\label{eq:double}
(e_{0}+1)B<m^{2}<(e_{0}+2)B.
\end{equation}
Dividing by $B>0$ gives
\begin{equation}\label{eq:unit-interval}
e_{0}+1<\frac{m^{2}}{B}<e_{0}+2.
\end{equation}

Next we claim that $\frac{m^{2}}{B}\notin\mathbb Z$, which implies that there exists a unique integer $e_0 \geq 0$, namely,    
\[
e_0= \left\lfloor \frac{m^{2}}{B}\right\rfloor -1 \ge 0
\] 
satisfying the Inequality~\eqref{eq:ineq-original}. To prove the last claim,  note that $B=m(m-h)+1\ge 2$, hence $1<\frac{m^{2}}{B}<m^{2}$.
Moreover,
\[
\gcd(m^{2},B)=\gcd(m^{2},m^{2}-mh+1)=\gcd(m^{2},mh-1)=1,
\]
since $\gcd(m,mh-1)=1$. 
Using the fact that  $2 \le B < m^2$, we see that $B$ does not divide $m^{2}$, so $\frac{m^{2}}{B}\notin\mathbb Z$.
\end{proof}

\begin{lemma}\label{lem:conversion}
For any coprime integers $0 < h < m$, we have 
$$
-D(2m^2-mh+1, m^2-mh+1)=Y\left(e_0; \dfrac{1}{2}, \dfrac{m^2+(e_0+1)(-m^2+mh-1)}{m^2-mh+1}, \dfrac{1}{2}\right),
$$ where $e_0$ is uniquely defined as in  Lemma~\ref{lem:unique-e0}. 
\end{lemma}

\begin{proof}
Using the notation from the previous proof we have
$$
D(2m^2-mh+1, m^2-mh+1)=D(m^2+B, B)
$$ 
where $B=m^{2}-mh+1$. By Lemma~\ref{lem:unique-e0}, there is a unique $e_0$ satisfying Inequality~\eqref{eq:unit-interval}. It follows that 
\begin{equation} \label{eq:s}
1+\dfrac{m^2}{B}=
\dfrac{B+m^2}{B}=[e_0+3, b_1, \ldots, b_n],
\end{equation} 
for some $b_i\geq 2$, and hence 
$$
D(m^2+B, B)=Y\left(-e_0-3; \dfrac{1}{2}, \dfrac{1}{[b_1, \ldots, b_n]}, \dfrac{1}{2}\right).
$$ 
By reversing the orientation, we get 
 $$
 -D(m^2+B, B)=Y\left(e_0; \dfrac{1}{2}, 1-\dfrac{1}{[b_1, \ldots, b_n]}, \dfrac{1}{2}\right).
 $$ 
 Setting
 $$s=1-\dfrac{1}{[b_1, \ldots, b_n]},$$ 
 we calculate using Equation~\eqref{eq:s} that 
 $$s=1-\left(e_0+2-\dfrac{m^2}{B}\right)=-e_0-1+\dfrac{m^2}{B}=\dfrac{m^2+(e_0+1)(-B)}{B}$$ 
as claimed in the lemma.
\end{proof}

\appendix

\section{Proof of Lemma 5.4.}\label{app}

In this appendix we give a proof of Lemma~\ref{lem: inverse} which calculates the inverse matrix $Q^{-1}_{{\bf a^1}, {\bf a^2}}$. 

 \begin{proof} [Proof of Lemma~\ref{lem: inverse}]  
We have $\widetilde{p}/\widetilde{q}=[a^1_1, a^1_2, \ldots, a^1_{n_1}]$,  $p/q=[a^2_0, a^2_1, \ldots, a^2_{n_2}]$, and $\widetilde{p}/(\widetilde{p}-\widetilde{q})=[a^3_1, a^3_2, \ldots, a^3_{n_3}]$, with all continued fraction coefficients $a^j_i$ being greater than or equal to $2$, by definition. Note that we also have, $\widetilde{p}/(\widetilde{q})^*=[a^1_{n_1}, a^1_{n_2}, \ldots, a^1_1]$, where $(\widetilde{q})^*$ is the inverse of $\widetilde{q}$ mod $\widetilde{p}$,    and $(p-q)/q=[a^2_0-1, a^2_1, \ldots, a^2_{n_2}]$ (see Remark~\ref{rem: specialcase}, when $a^2_0=2$). 

We observe the fact that ${\bf v}^B={\bf v}^{\widetilde{B}}$, which immediately follows from the definitions of the matrices $B$ and $\widetilde{B}$.  By Lemma~\ref{lem: firstrow},  we know that the last column of $A^{-1}$ is  $-(1/\widetilde{p})\;  {\bf u}^A$,   the first column of  $(\widetilde{B})^{-1}$ is  $-(1/(p-q)) \; {\bf v}^B$ (and hence the first row of 
$(\widetilde{B})^{-1}$ is given by $-(1/(p-q)) \; ({\bf v}^B)^T$) and  the first column of  $C^{-1}$ is  $-(1/\widetilde{p})\;  {\bf v}^C$.

The proof of the lemma will be achieved by showing that the product of the matrix $Q_{{\bf a^1}, {\bf a^2}}$, depicted in Figure~\ref{fig: matrix}, and the matrix depicted in Figure~\ref{fig: inversematrix} is equal to the identity matrix. Towards that goal, we first define an auxiliary matrix $\widetilde{Q}_{{\bf a^1}, {\bf a^2}}$ by deleting $A^{-1}$ and $C^{-1}$ from the matrix depicted in Figure~\ref{fig: inversematrix}, and calculate the product of $Q_{{\bf a^1}, {\bf a^2}}$ and  $\widetilde{Q}_{{\bf a^1}, {\bf a^2}}$. 

Let $\widetilde{Q}_A$ be the submatrix of $\widetilde{Q}_{{\bf a^1}, {\bf a^2}}$ obtained by juxtaposition of the blocks $G, D$, and $E$, let $\widetilde{Q}_{\widetilde{B}}$ be the submatrix of  $\widetilde{Q}_{{\bf a^1}, {\bf a^2}}$  obtained by juxtaposition of the blocks  $D^T, (\widetilde{B})^{-1}$ and $F$ and let $\widetilde{Q}_C$ be the submatrix of  $\widetilde{Q}_{{\bf a^1}, {\bf a^2}}$ obtained by juxtaposition of the blocks  $E^T, F^T$, and $H$.

By Remark~\ref{rem: row}, we know that each column of $\widetilde{Q}_A$ is a multiple of ${\bf u}^A$ (because  the definition of $G, D, E$ involves ${\bf u}^A$), that each column of $\widetilde{Q}_C$ is a multiple of ${\bf v}^C$ (because  the definition of $E^T, F^T, H$  involves ${\bf v}^C$) and that the columns of $\widetilde{Q}_{\widetilde{B}}$ belonging to the blocks $D^T$ and  $F$  is a multiple of ${\bf v}^B$ (because the definition of $D^T$ and  $F$  involves ${\bf v}^B$). 
Note that the first column of the block $(\widetilde{B})^{-1}$ is also a multiple of ${\bf v}^B$.   

Now we calculate the dot product of each row of $Q_{{\bf a^1}, {\bf a^2}}$ with each column of  $\widetilde{Q}_{{\bf a^1}, {\bf a^2}}$ case by case below. Note that for each dot product there are at most four terms to calculate. So, the calculations are not very difficult but tedious. 

For $1 \leq i \leq n_1-1$, and $1 \leq j \leq n_1+n_2+n_3+1$, the dot product of the $i$th row of $Q_{{\bf a^1}, {\bf a^2}}$  and the $j$th  column of $\widetilde{Q}_{{\bf a^1}, {\bf a^2}}$ is the same as the dot product of the $i$th row of $A$ and the  $j$th  column of $\widetilde{Q}_A$. Hence this is the dot product of the $i$th row of $A$ and a multiple of the last column of $A^{-1}$ since the last column of $A^{-1}$ is a multiple of ${\bf u}^A$ and each column of $\widetilde{Q}_A$ is a multiple of ${\bf u}^A$. Therefore this dot product is zero, because we are indeed taking the dot product of any of the first $n_1-1$ rows of $A$ with a multiple of the last column of $A^{-1}$.  

We want to show  that for $1 \leq j \leq n_1+n_2+n_3+1$, the dot product of the $n_1$th row of $Q_{{\bf a^1}, {\bf a^2}}$ (the last row of $\widetilde{Q}_A$) and the $j$th  column of $\widetilde{Q}_{{\bf a^1}, {\bf a^2}}$ is zero. To see this,  we first claim that the $(n_1+1)$st row of  $\widetilde{Q}_{{\bf a^1}, {\bf a^2}}$ (the first row of $\widetilde{Q}_{\widetilde{B}}$) is $(\det A=\widetilde{p})$ multiple of the first row of $\widetilde{Q}_{{\bf a^1}, {\bf a^2}}$. To prove the claim we observe that the first row of $(\widetilde{B})^{-1}$ is $\frac{-1}{p-q} ({\bf v}^B)^T=\frac{-1}{p-q} ({\bf v}^{\widetilde{B}})^T$ and the first entry of ${\bf v}^B$ is $q$. Then the  claim follows by simply comparing the first rows of  $D^T, (\widetilde{B})^{-1}, F$   with the first rows of $G, D, E,$ respectively as follows. The first row of $$D^T = -\dfrac{1}{\widetilde{p} (p-q)}   {\bf v}^B ({\bf u}^A)^T \; \mbox{is given by} \;    -\dfrac{q}{\widetilde{p} (p-q)} ({\bf u}^A)^T$$ and the first row of $$G= -\dfrac{q}{(\widetilde{p})^2(p-q)}  {\bf u}^A ( {\bf u}^A)^T \; \mbox{is given by} \;   -\dfrac{q}{(\widetilde{p})^2(p-q)} ( {\bf u}^A)^T.$$ As pointed out above, the first row of 
$$(\widetilde{B})^{-1}  \; \mbox{is given by} \; \frac{-1}{p-q} ({\bf v}^B)^T$$ and the first row of $$D = -\dfrac{1}{\widetilde{p} (p-q)}  {\bf u}^A ({\bf v}^B)^T   \; \mbox{is given by} \;    -\dfrac{1}{\widetilde{p} (p-q)} ({\bf v}^B)^T.$$ Similarly,  the first row of $$E = -\dfrac{q}{(\widetilde{p})^2(p-q)}   {\bf u}^A  ({\bf v}^C)^T \; \mbox{is given by} \;    -\dfrac{q}{(\widetilde{p})^2(p-q)} ({\bf v}^C)^T$$ and the first row of $$F=  -\dfrac{1}{\widetilde{p} (p-q)}   {\bf v}^B ({\bf v}^C)^T \; \mbox{is given by} \;   -\dfrac{q}{\widetilde{p} (p-q)} ({\bf v}^C)^T,$$ which finishes the proof of our claim. Note that the $n_1$th row of $Q_{{\bf a^1}, {\bf a^2}}$  can be seen as the  juxtaposition of the last row of $A$ and the row vector $[1, 0, \dots, 0] \in \mathbb{R}^{n_2+n_3+1}$. So, in the dot product at hand we only have to consider the first $(n_1+1)$st entries  in the $j$th column of $\widetilde{Q}_{{\bf a^1}, {\bf a^2}}$. Our analysis above therefore implies that for $1 \leq j \leq n_1+n_2+n_3+1$, the dot product of the $n_1$th row of $Q_{{\bf a^1}, {\bf a^2}}$ and the $j$th column of $\widetilde{Q}_{{\bf a^1}, {\bf a^2}}$ is equal to a multiple of $$(\mbox{last row of}\; A) \cdot {\bf u}^A + \det A$$ but ${\bf u}^A$ is $-\det A$ times the last column of $A^{-1}$, and therefore the dot product is a multiple of 
$$-\det A (\mbox{last row of}\; A) \cdot (\mbox{last column of}\; A^{-1} )+ \det A=0$$
 in this case as well.

Now we turn our attention to the last $n_3-1$  rows of $Q_{{\bf a^1}, {\bf a^2}}$. For $$n_1+n_2+3 \leq i \leq n_1+n_2+n_3+1\; \mbox{and}\; 1 \leq j \leq n_1+n_2+n_3+1,$$ the dot product of the $i$th row of $Q_{{\bf a^1}, {\bf a^2}}$ and $j$th  column of $\widetilde{Q}_{{\bf a^1}, {\bf a^2}}$ is the same as the dot product of the $(i-n_1-n_2-1)$st  row of $C$ and the  $j$th  column of $\widetilde{Q}_C$. Hence this is the dot product of the $(i-n_1-n_2-1)$st row of $C$ and a multiple of the first column of $C^{-1}$ since the first column of $C^{-1}$ is a multiple of ${\bf v}^C$ and  each column of  $\widetilde{Q}_C$ is a multiple of ${\bf v}^C$.  Therefore this dot product is zero,  because we are indeed taking the dot product of any of the last  $n_3-1$ rows of $C$ with a multiple of the first column of $C^{-1}$. 

For  $1 \leq j \leq n_1+n_2+n_3+1$, the dot product of the $(n_1+n_2+2)$nd row of $Q_{{\bf a^1}, {\bf a^2}}$ and the $j$th  column of $\widetilde{Q}_{{\bf a^1}, {\bf a^2}}$ is zero because of the fact that the $(n_1+1)$st row of  $\widetilde{Q}_{{\bf a^1}, {\bf a^2}}$ (the first row of $\widetilde{Q}_{\widetilde{B}}$) is $(\det A =\widetilde{p}=\det C)$ multiple of the first row of $\widetilde{Q}_{{\bf a^1}, {\bf a^2}}$, as we showed above.   Hence, the dot product at hand is equal to a multiple of $$(\mbox{first row of}\; C) \cdot {\bf v}^C + \det C$$ but ${\bf v}^C$ is $-\det C$ times the first column of $C^{-1}$, and therefore the dot product is zero in this case as well.

Now we turn our attention to the middle rows of  $Q_{{\bf a^1}, {\bf a^2}}$. For $n_1+2 \leq i \leq n_1+n_2+1$ and ($1 \leq j \leq n_1+1$ and $n_1+n_2+2 \leq j \leq n_1+n_2+n_3+1$) the dot product of the $i$th row of $Q_{{\bf a^1}, {\bf a^2}}$ and the $j$th  column of $\widetilde{Q}_{{\bf a^1}, {\bf a^2}}$ is the same as the dot product of the $(i-n_1)$th row of $B$ and the  $j$th  column of  $\widetilde{Q}_{\widetilde{B}}$. Hence this is the dot product of the $(i-n_1)$th row of $B$ and a multiple of the first column of $(\widetilde{B})^{-1}$ since the first column of $(\widetilde{B})^{-1}$ is a multiple of ${\bf v}^{\widetilde{B}}={\bf v}^B$ and every column except the ones enumerated from $n_1+2$ to $n_1+n_2+1$ of  $\widetilde{Q}_{\widetilde{B}}$ is a multiple of ${\bf v}^B$. Therefore this dot product is zero, because we are taking the dot product of any of the last $n_2$ rows of $B$ (same as the last $n_2$ rows of  $\widetilde{B}$) with a multiple of the first column of $(\widetilde{B})^{-1}$. 

For $n_1+2 \leq i \leq n_1+n_2+1$ and $n_1+2 \leq j \leq n_1+n_2+1$,  the dot product of the $i$th row of $Q_{{\bf a^1}, {\bf a^2}}$  and the $j$th  column of $\widetilde{Q}_{{\bf a^1}, {\bf a^2}}$ gives a $n_2 \times n_2$ identity matrix because these dot products are the same as the dot products of last $n_2$  rows of $\widetilde{B}$ and the last $n_2$ columns of  $(\widetilde{B})^{-1}$. 

In order to move forward with our analysis, we prove a general fact that will be used below. Since each column of $\widetilde{Q}_C$ is a multiple of ${\bf v}^C$, and the first entry of ${\bf v}^C$ is $(\widetilde{p}-\widetilde{q})$, while its last entry is $1$, we see that the first row of $\widetilde{Q}_C$ is $(\widetilde{p}-\widetilde{q})$ multiple of its last row. Recall that we showed above that the first  row of  $\widetilde{Q}_{\widetilde{B}}$ is $\widetilde{p}$ multiple of the first row of $\widetilde{Q}_A$. Since each column of $\widetilde{Q}_A$ is a multiple of ${\bf u}^A$, and the first entry of ${\bf u}^A$ is $1$, while its last entry $\widetilde{q}$, we see that the last row of $\widetilde{Q}_A$ is $\widetilde{q}$ multiple of its first row. By simply comparing the definitions of the block matrices involved, we also observe that the first row of $\widetilde{Q}_A$ is the same as the last row $\widetilde{Q}_C$. Putting all these observations together, we conclude that 
\begin{equation}\label{eq: abcrows}
 \mbox{last row of}\;  \widetilde{Q}_A +  \mbox{first row of}\;  \widetilde{Q}_C = \mbox{first row of}\;  \widetilde{Q}_{\widetilde{B}} 
\end{equation}
which is a key result for the rest of our proof. 

Next we want to see that the dot product of the $(n_1+1)$st row of $Q_{{\bf a^1}, {\bf a^2}}$ and the  $(n_1+1)$st  column of $\widetilde{Q}_{{\bf a^1}, {\bf a^2}}$ is $1$. The equality~(\ref{eq: abcrows}) can be rephrased as 
\begin{equation}\label{eq: abcrows2}
(n_1)\mbox{th row} + (n_1+n_2+2)\mbox{nd row} = (n_1+1)\mbox{th row} 
\end{equation}
for the matrix $Q_{{\bf a^1}, {\bf a^2}}$. Note that  in the  $(n_1+1)$st  row of $Q_{{\bf a^1}, {\bf a^2}}$, the nonzero terms appear on the $n_1$th,  $(n_1+1)$st , $(n_1+2)$nd and $(n_1+n_2+2)$nd entries as $1,-a_0^2,1,1$. Because of the equality~(\ref{eq: abcrows2}),   the dot product at hand is exactly the dot product of the first row of $\widetilde{B}$ (which is the vector $[-(a_0^2-1) \; 1 \; 0 \cdots 0] \in \mathbb{R}^{n_2+1}$ and the first column of $(\widetilde{B})^{-1}$, which is equal to $1$. 

What is left to consider is the dot product of the  $(n_1+1)$st  row of $Q_{{\bf a^1}, {\bf a^2}}$ and the  $j$th column of  $\widetilde{Q}_{{\bf a^1}, {\bf a^2}}$ for  $j \neq n_1+1$. If we take $1 \leq j \leq n_1$, the dot product of the $(n_1+1)$st row of $Q_{{\bf a^1}, {\bf a^2}}$  and the $j$th column of $\widetilde{Q}_{{\bf a^1}, {\bf a^2}}$ is not zero. The computation here is very similar to the case $j=n_1+1$ we discussed in the paragraph above. In fact the dot product here is equal to some multiples (for $1 \leq j \leq n_1$) of $1$ (that we  computed in the previous paragraph), and as a matter of fact  these dot products are exactly given by the row vector $(1/\widetilde{p})\; ({\bf u}^A)^T$. Therefore,  when we compute the dot product  of the $(n_1+1)$st  row of $Q_{{\bf a^1}, {\bf a^2}}$  and 
the $j$th column of claimed $Q_{{\bf a^1}, {\bf a^2}}^{-1}$ (where we have to take into account $A^{-1}$ at the top left block) we just have to add $-(1/\widetilde{p})\; ({\bf u}^A)^T$ (the last row of $A^{-1}$) and hence the dot product will be zero.  

Similarly, if we take $n_1+n_2+2 \leq j \leq n_1+n_2+n_3+1$, the dot product of the $(n_1+1)$st row of $Q_{{\bf a^1}, {\bf a^2}}$  and the $j$th column of $\widetilde{Q}_{{\bf a^1}, {\bf a^2}}$ is not zero and in fact equal to some multiples of $1$ we computed above. These multiples are exactly given by the row vector $(1/\widetilde{p})\; ({\bf v}^C)^T$. Hence ,  for $n_1+n_2+2 \leq j \leq n_1+n_2+n_3+1$, when we compute the dot product  of the $(n_1+1)$st row of $Q_{{\bf a^1}, {\bf a^2}}$  and the $j$th column of claimed $Q_{{\bf a^1}, {\bf a^2}}^{-1}$ (where we have to take into account $C^{-1}$ at the bottom right block) we just have to add $-(1/\widetilde{p})\; ({\bf v}^C)^T$ (the first row of $C^{-1}$) and hence the dot product will be zero.

Finally, we can see that, for  $n_1+2 \leq j \leq n_1+n_2+1$, the dot product of the $(n_1+1)$st row of $Q_{{\bf a^1}, {\bf a^2}}$  and the $j$th column of $\widetilde{Q}_{{\bf a^1}, {\bf a^2}}$ 
 is zero. The key point is that these are some multiples of  the dot products of the first row of $\widetilde{B}$ with the last $n_2$ columns of $(\widetilde{B})^{-1}$. Here we again use  the equality~(\ref{eq: abcrows2}). 

As a consequence of our case by case analysis  above,  we can explicitly describe the product $Q_{{\bf a^1}, {\bf a^2}}$ and $\widetilde{Q}_{{\bf a^1}, {\bf a^2}}$ as follows. It is a square matrix of size $n_1+n_2+n_3+1$ such that the first $n_1$ and last $n_3$ rows are zero. The middle rows can be described as the juxtaposition of    three block matrices of sizes $(n_2+1) \times n_1$,  $(n_2+1) \times (n_2+1)$, $(n_2+1) \times n_3$, respectively.  The first row of the first block is $(1/\widetilde{p})\; ({\bf u}^A)^T$ and all other rows are zero. The second block is the $(n_2+1) \times (n_2+1)$ identity matrix and the first row of the third block is 
$(1/\widetilde{p})\; ({\bf v}^C)^T$ and all other rows are zero.

Recall that $\widetilde{Q}_{{\bf a^1}, {\bf a^2}}$ is obtained by deleting $A^{-1}$ and $C^{-1}$ from the matrix depicted in Figure~\ref{fig: inversematrix}. Therefore, based on the previous paragraph and the fact that  $$AA^{-1}=I\;  \mbox{and} \; CC^{-1}=I,$$ we see that the product of the matrix $Q_{{\bf a^1}, {\bf a^2}}$ depicted in Figure~\ref{fig: matrix} and the matrix depicted in Figure~\ref{fig: inversematrix} is equal to the identity matrix. This is because the nonzero row  $(1/\widetilde{p})\; ({\bf u}^A)^T$ in the previous paragraph will be cancelled out by the last row of $A^{-1}$, which is equal to $-(1/\widetilde{p})\; ({\bf u}^A)^T$ and similarly, the nonzero row  $(1/\widetilde{p})\; ({\bf v}^C)^T$ in the previous paragraph will be cancelled out by the first row of $C^{-1}$, which is equal to $-(1/\widetilde{p})\; ({\bf v}^C)^T$. We would like to emphasize that we have also used the fact that $\widetilde{B}(\widetilde{B})^{-1}=I$, while calculating the product of $Q_{{\bf a^1}, {\bf a^2}}$ and $\widetilde{Q}_{{\bf a^1}, {\bf a^2}}$. 
\end{proof} 

\bibliography{references}
\bibliographystyle{plain}
\end{document}